\documentclass[reqno]{amsart}
\usepackage{a4wide}
\usepackage{amsmath,amssymb,amsthm,graphicx}
\usepackage{latexsym}
\usepackage{amscd}
\usepackage{amsfonts}
\usepackage[usenames,dvipsnames]{color}
\usepackage{verbatim}
\usepackage{float}
\usepackage{xypic}
\usepackage[pdftex,colorlinks]{hyperref}

\DeclareMathOperator{\Int}{Int}
\DeclareMathOperator{\length}{length}
\DeclareMathOperator{\Ker}{Ker}

\DeclareMathOperator{\Fix}{Fix}

\DeclareMathOperator{\Bil}{Bil}

\DeclareMathOperator{\Isom}{Isom}
\DeclareMathOperator{\arcsinh}{arcsinh}

\newtheorem{prop}{Proposition}[section]
\newtheorem{theorem}[prop]{Theorem}

\newtheorem{lem}[prop]{Lemma}
\newtheorem{cor}[prop]{Corollary}
\theoremstyle{definition}

\newtheorem{definition}[prop]{Definition}
\newtheorem{rem}[prop]{Remark}
\newtheorem{example}[prop]{Example}

\newtheorem*{rem*}{Remark}
\newtheorem*{notation*}{Notation}

\newtheorem{defi}[prop]{Definition}

\numberwithin{equation}{section}

\newcommand{\bR}{\mathbb{R}}
\newcommand{\ra}{\rightarrow}

\newcommand{\cB}{\mathcal{B}}

\newcommand{\End}{\mathrm{End}}
\DeclareMathOperator{\id}{id}

\DeclareMathOperator{\round}{round}

\newcommand{\Hom}{\mathrm{Hom}}

\newcommand{\TPf}{\mathrm{TPf}}

\newcommand{\II}{\mathrm{II}}
\newcommand{\Pf}{\mathrm{Pf}}
\newcommand{\oPf}{\Pf^{\mathrm{odd}}}
\newcommand{\pM}{\partialM }
\newcommand{\scal}{\mathrm{scal}}

\newcommand{\vol}{\mathrm{vol}}

\newcommand{\HN }{H\! N}
\newcommand{\VN }{V\! N}
\newcommand{\HP }{H\! P}
\newcommand{\VP }{V\! P}
\newcommand{\nuB}{\nu\! B }

\newcommand{\partialM}{\partial\! M }

\begin{document}

\title{Odd Pfaffian forms}
\author{Daniel Cibotaru}
\thanks{Partially supported by the CNPq Universal Project.}
\address{Universidade Federal do Cear\'a, Fortaleza, CE, Brazil}
\email{daniel@mat.ufc.br} 

\author{Sergiu Moroianu}
\thanks{Partially supported by the CNCS project PN-III-P4-ID-PCE-2016-0330.}
\address{Institutul de Matematic\u{a} al Academiei Rom\^{a}ne\\
P.O. Box 1-764\\ 
RO-014700 Bucharest\\
Romania}
\email{moroianu@alum.mit.edu}
\date{\today}

\begin{abstract} On any odd-dimensional oriented Riemannian manifold we define a
volume form called the \emph{odd Pfaffian} through a certain 
invariant polynomial with integral coefficients in the curvature tensor.
We prove an intrinsic Chern-Gauss-Bonnet formula for incomplete edge
singularities in terms of the odd Pfaffian on the fibers of the boundary fibration.
The formula holds for product-type model edge metrics where the degeneration 
is of conical type in each fiber, but also for  perturbations of second order 
of the model metrics. The same method produces a Chern-Gauss-Bonnet formula 
for complete, non-compact manifolds with fibered boundaries in the sense 
of Mazzeo-Melrose and perturbations thereof, this time involving the 
odd Pfaffian of the base (rather than the fiber) of the fibration. 
We deduce the rationality of the usual Pfaffian form on Riemannian orbifolds, 
and exhibit obstructions for certain metrics on a fibration to be realized 
as the model at infinity of a flat metric with conical, edge or fibered 
boundary singularities.
\end{abstract}
\maketitle


\section{Introduction}

Gauss-Bonnet formulas in singular geometric contexts abound in
mathematical literature, we mention here for instance \cite{A,AW, BrK,
CMSI,CMSII, DW,Du,Fu, McM, Ra, Ro, W}.
With a few notable exceptions, most of those theorems treat the case of
singular sets embedded in a smooth Riemannian manifold $M$, typically $M=\bR^n$,
since by the Nash embedding theorem all Riemannian manifolds are isometrically
embeddable in some euclidean space. In this article 
we look at a different type of degeneration, for which the techniques of the
"embedded" situation do not apply. Namely, we consider a compact differentiable
manifold $M$ with boundary, endowed with a Riemannian metric which is smooth in
the interior and degenerates at the boundary following certain precise patterns.
Examples of such degenerate metrics include the so called \emph{incomplete edge
metrics}, for instance any Riemannian metric in the complement of a submanifold, 
and also the \emph{fibered boundary metrics}, a class of complete metrics
including the generalized Taub-NUT metrics on $\bR^4$. 

\subsection*{Double forms and the odd Pfaffian}
We set the stage with our own algebraic treatment of the Gauss-Bonnet formula on
compact oriented manifolds $(M^{2k},g)$ using the formalism of double forms:
\begin{align*}(2\pi)^k\chi(M)=\int_M\Pf(g),&&\Pf(g)=\frac{1}{k!}
\cB_g\left((R^g)^k \right).
\end{align*}
Here $R^g\in\Lambda^2(M)\otimes\Lambda^2(M)$ is the curvature
form of $g$, a double form of bi-degree $(2,2)$, and $\cB_g$ is the Berezin
integral, or contraction with the volume form of $g$ in the second component. 
When $M$ has a nonempty boundary $(N,h)$, essentially as a consequence 
of the second Bianchi identity we isolate a correction term 
when the metric is not of product-type near the boundary:
\begin{align}\label{eq-4}
(2\pi)^k\chi(M)=\int_M\Pf(g)-
\sum_{j=0}^{k-1}
\frac{(-1)^{k+j}(2k-2j-3){!}{!}}{j!(2k-2j-1)!}
 \int_{\partialM }\cB_h\left((R^h)^{j}\wedge \II^{2k-2j-1} \right).
\end{align}
In this formula $\II\in\Lambda^1(N)\otimes\Lambda^1(N)$ is the second
fundamental
form of the boundary, a double form of bi-degree $(1,1)$, $\cB_h$ is the Berezin
integral with respect to $h$, and
\begin{align*}
(-1){!}{!}:=1,&&
(2n-1){!}{!}:=1\cdot 3\cdot\ldots\cdot(2n-1)\text{ for $n\geq 1$}.
\end{align*}
Of course, in coordinates this coincides with the correction term 
of the original formul{\ae} of \linebreak Allendoerfer-Weil \cite{AW} and Chern
\cite{Ch1,Ch2}.
This compact algebraic way of writing the Gauss-Bonnet
integrand on the boundary is well-suited for generalizations.

Motivated by \eqref{eq-4},
we define the \emph{odd Pfaffian form} of a $2k-1$-dimensional Riemannian
manifold $(N,h)$
in terms of the curvature form $R^h\in\Lambda^2\otimes\Lambda^2$ and the metric
tensor $h\in\Lambda^1\otimes\Lambda^1$. 
\begin{defi}
For every oriented $2k-1$-dimensional Riemannian manifold $(N,h)$ define
\[\oPf(h):=\sum_{j=0}^{k-1} (-1)^{k+j}(2k-2j-3){!}{!}\cB_h \left(
\frac{(R^h)^j\wedge h^{2k-1-2j}}{j!(2k-2j-1)!}\right)\in\Lambda^{2k-1}(N).
\]
\end{defi}
In any orthonormal frame, $\oPf$ is a polynomial with integral coefficients in
the entries of the curvature form $R$. Up to a constant, 
this form appears already, in a different presentation, in the work of Albin
\cite[Eq.\ (7.12)]{A} as the boundary correction term in the Gauss-Bonnet
formula for scattering metrics. It consists of a linear combination
with integral coefficients of the  Lipschitz-Killing curvatures (Definition
\ref{LKcu}). As explained in Section \ref{conman}, the odd Pfaffian  is in fact
the transgression of the Pfaffian for any slice $\{r\}\times N$
on the cone $(-\epsilon,0)\times N$ with the metric $dr^2\oplus r^2 h$.

\subsection*{Edge singularities}
The first type of metrics analyzed here are the incomplete edge metrics. This
means 
we have an (oriented) compact manifold with
boundary $M$ together with a fibration structure of the boundary $\pi:\partialM
\to B$ over a compact manifold $B$.
Fix a boundary-defining function $r$ for the boundary.
The (singular) metric in a collar neighborhood  of 
$\partialM =\{r=0\}$ has the form
\begin{align}\label{eq-3} 
g=dr^2\oplus g(r),&&g(r)=r^2g^{V}\oplus \pi^*g^B
\end{align}
where $g^B$ is a metric on $B$, $g^V$ is a Riemannian metric on the fibers and
the splitting is induced by an Ehresmann connection. Even in this first analysis
we allow $g^V$  to vary with $r$ but still converging to some true
metric at $r=0$.

We prove that a Gauss-Bonnet formula holds on such manifolds and we
compute the contribution of the singular locus $\partialM $ in terms
of the geometric data, essentially the Pfaffian of the base and the odd Pfaffian
of the fibers. 
Due to its importance in geometric applications, we review the (perturbed)
conical case separately 
(see Theorem \ref{Tconca}). 

\begin{theorem}\label{Theorem1} Let $(M^{2k},g)$ be a manifold with edge
singularities with $g$ as in \eqref{eq-3}. 
\begin{itemize}
\item[(a)] If $\dim(B)$ is odd,
\[ \chi(M)=\frac{1}{(2\pi)^k}\int_M\Pf^g.
\]
\item[(b)] If $\dim(B)$ is even,
\[
(2\pi)^k
\chi(M)=\int_{M}\Pf^g
-\int_B\left(
{\Pf(g^B)}\int_{\partialM /B}\oPf(g^V)\right).\]
\end{itemize}
\end{theorem}

When we allow horizontal variations of the metric, i.e., $g^B$ varies with $r$,
we obtain certain additional
terms (see Theorem \ref{olst}).

The computation is based on two observations. First, the second fundamental
form of a slice is the Lie derivative of the metric in the direction of
the normal geodesic flow $\partial_r$. Secondly, we describe 
explicitly the decomposition of
the curvature form of a Riemannian submersion into its horizontal, mixed and
vertical components with respect to the second variable when seen as a double
form.

\subsection*{Manifolds with fibered boundaries}
The same method used for edge metrics leads to a
Gauss-Bonnet formula for a different type of degeneracy.
Following Mazzeo and Melrose \cite{MaMe3}, 
a  non-compact Riemannian manifold $(M,g)$ is called with \emph{fibered
boundary} 
if it has  a finite number of
ends which are modeled on $(1,\infty)\times N$ with the metric
\[ g:= dr^2\oplus g^V\oplus r^2\pi^*g^B
\]
for $r\gg 1$.
We assume here that $N\ra B$ is a fiber bundle with a fixed Ehresmann
connection with respect to which the extension of $g^V$ to $N$ is defined. It is
not hard to see that 
such a metric is complete. (These metrics were studied in depth by Vaillant in
\cite{Va} 
under the name $\phi$-metrics.) Let $F$ be a generic fiber of $\pi$, $b:=\dim B$
and $f:=2k-1-b$ the dimension of $F$. 
\begin{theorem} \label{th1.3}
Let $(M^{2k},g)$ be a manifold with fibered boundary. 
\begin{itemize}
\item[(a)] If $b$ is even,
\[ \chi(M)=\frac{1}{(2\pi)^k}\int_{M}\Pf^g.\]
\item[(b)] If $b$ is odd,
\begin{align} 
(2\pi)^k \chi(M)=&\int_{M}\Pf^g+(2\pi)^{{f/2}}\chi(F)\int_{B}\oPf(g^B) 
\label{gbt}.
\end{align}
\end{itemize}
\end{theorem}
Compared with Theorem \ref{Theorem1} there are two significant differences: the 
odd Pfaffian appears now in the base, not in the fibers; and the sign in front
of the transgression has changed.

The Gauss-Bonnet problem for fibered boundary metrics was previously studied by
Albin \cite{A} 
and also by Dai-Wei \cite{DW}. Theorem \ref{th1.3} can be seen as an extension
of 
their partial results.
Albin gives a formula in the case where either the fiber or the base of the
boundary 
fibration reduce to a point, while for $\dim(M)=4$, Dai and Wei give the
formula 
when the fiber is a point, i.e., for  "large conical" metrics, better known as 
scattering metrics by the Melrose school.
Note that Dai-Wei also state a formula in the general case, claiming the 
vanishing of the transgression term from \eqref{gbt}. This claim holds true for
even-dimensional $B$, 
but is incorrect when the base is odd-dimensional, as noted also in \cite{Ze}.
(They apply this result in dimension four 
when the fiber is a circle, hence their results concerning Hitchin-Thorpe
inequalities on blow-ups 
of the Taub-NUT space are not adversely affected by this issue.)

\subsection*{Perturbations of the model degenerate metrics and transgressions}
The Chern-Gauss-Bonnet formul{\ae} for incomplete edge metrics and for fibered
boundary metrics stated above 
in terms of the odd Pfaffian are new, even in the model case. 
In the context set forth in this paper we should mention, besides the thesis
paper of Albin cited above, 
previous results obtained by Rosenberg \cite{Ro} and Grieser \cite{Gr2}.  The main statement from  \cite{Ro} can be
seen as a particular case of 
Theorem \ref{th1.3}.  Our "conical" Gauss-Bonnet Theorem \ref{Tconca} recovers  Theorem 1.4 from \cite{Gr2}, albeit with a slightly stronger differentiability condition on the metric. 
 
In our view, one of the pleasant results of this work is being able to extend the
results from model metrics to large classes of perturbations of the model metrics
$g$ described in \eqref{eq-3}. We show that if the
perturbations of $g$ are of second order, in a sense made precise in Def.\
\ref{def10}, the formul{\ae} from Theorems \ref{Theorem1} and \ref{th1.3}
remain valid. 

It turns out that when one deals with (product-type) model metrics, one can take
advantage of certain symmetries in order to perform the computations, 
like being able to isolate the various components of the curvature form 
and second fundamental form. This does not seem to be case when 
perturbations are allowed, raising some difficulties for a direct 
computational approach.

In compensation, properties of transgression forms are fundamental for the
proofs given here and allow us to use arguments of topological nature 
in places where geometric computations seem overly complicated. 
We devote a first section to proving such properties, since they 
are not part of mainstream presentation of Chern-Weil theory.

Recall that given an Euclidean vector bundle $E\ra B$ of rank $2k$ 
endowed with two metric connections $\nabla_1$, $\nabla_2$, 
there exists a canonical form $\TPf(\nabla_1,\nabla_2)$ satisfying
\[ \Pf(\nabla_1)-\Pf(\nabla_2)=d\TPf(\nabla_1,\nabla_2).
\]
It is known since Chern \cite{Ch2} that the boundary integrand in the 
standard
Gauss-Bonnet Theorem can be described as such a transgression form. So at
first it might seem unremarkable that the correction term in Gauss-Bonnet
Theorem for first-order perturbations (see below) of the model metric is a
transgression
form integrated over the boundary. However, one should keep in mind that due to
the degeneracy of the metric, there is \emph{a priori} no well-defined 
connection along the singular locus, let alone two of them.

We  analyze perturbations of the model degenerate metrics, both 
for incomplete edge metrics and for complete fibered boundary metrics. 
The methods to treat the two cases
are similar and we only outline here the treatment of the non-complete (edge) case. One
natural approach would be to follow the ideas first introduced by
Melrose in the general context of the $b$-calculus \cite{Me, MeW}, and employ as 
background the edge tangent bundle, transferring all geometric structures
onto it. Nevertheless, since the edge tangent bundle is isomorphic 
(albeit non-canonically) to the tangent bundle, rather than relying explicitly
on this natural notion we prefer
to work here with an endomorphism $\varphi\in \End(TM)$ which has, 
given the choice of a
boundary defining function $r$, the following expression in a
collar neighborhood of $\partialM $:
\[ \varphi(v,w)=(rv,w),
\]
i.e., $\varphi$ acts as multiplication by $r$ on the vertical 
component of the fiber bundle $\partialM \ra B$ and leaves the horizontal and
the normal components unchanged. 
(Of course, the edge tangent bundle remains hidden behind the curtain.)

The endomorphism $\varphi$ is an isomorphism in the interior but not at $r=0$.
It is easy to see that the pull-back  
\[g^{\varphi}(\cdot,\cdot):=g(\varphi^{-1}(\cdot),\varphi^{-1}(\cdot))\]
of the model degenerate metric $g$ extends to a smooth metric on $TM$. 
Consequently, we consider perturbations $\tilde{g}$ of
$g$ that preserve this property. In fact, a perturbation $\tilde{g}$  of $g$ is
a degenerate metric that satisfies 
\[\tilde{g}^{\varphi}=g^{\varphi}+\alpha(\cdot,\cdot)\] for certain smooth
symmetric bilinear 
form $\alpha$ which vanishes at least to order $1$ at $r=0$. 
We call the perturbation to be of order $j\geq 1$ if $\alpha\in O(r^j)$. 

The main result that allows the investigation of Gauss-Bonnet formulas for
perturbations of model metrics is the next theorem which should be compared
with extension results for the Levi-Civita
connection in the context of $\phi$-geometry (see \cite{Va}, Prop. 1.5).
\begin{theorem}\label{theorem02} Let $\nabla^g$, $\nabla^{\tilde{g}}$ be the
Levi-Civita
connections of the edge degenerate metric $g$ and a first-order perturbation
$\tilde{g}$. Then $\varphi\nabla^g\varphi^{-1}$ and
$\varphi\nabla^{\tilde{g}}\varphi^{-1}$ extend to smooth connections on $TM$. 
If $\tilde{g}$ is a second-order perturbation, then the restriction of these
connections to $r=0$ coincide:
\begin{equation}\label{eq-6}
\varphi\nabla^{\tilde{g}}\varphi^{-1}\bigr|_{r=0}=\varphi\nabla^{{g}}\varphi^{-1
}\bigr|_{r=0}.
\end{equation}
\end{theorem}

We use an "abstract" version of the Christoffel coefficients formula which
reduces this theorem to proving the smooth extension at $r=0$
of the Levi-Civita connection for the model metric $g$. 
It is exactly property \eqref{eq-6} that allows one to conclude that Theorem
\ref{Theorem1} holds  for second-order perturbations.

A consequence that is easy to miss of Theorem \ref{theorem02} is that even for
first-order perturbations $\tilde{g}$ of the model metrics $g$ one still has a
Gauss-Bonnet formula of type
\begin{equation}\label{eqintro3}
(2\pi)^k\chi(M)=\int_{M}\Pf^{\tilde{g}}+\int_B\gamma
\end{equation}
for some geometric term $\gamma$ which is itself the result of integration over
the fibers $\partial M\ra B$ of a geometric quantity which takes the guise of a
transgression form as follows. Let 
\begin{align*}
\nabla^1:={}&\varphi\nabla^{\tilde{g}}\varphi^{-1}\bigr|_{r=0},&
\nabla^0:={}&\varphi\nabla^{{g}}\varphi^{-1}\bigr|_{r=0}.
\end{align*}
be the two connections on $TM\bigr|_{\partial M}$ whose existence is guaranteed 
by Theorem \ref{theorem02}. The restriction $\nabla^0$ has a particularly 
simple geometric description (see Corollary \ref{Cor72}). 
Then the following Gauss-Bonnet formula holds:

\begin{theorem}\label{Theorem5}  Let $\tilde{g}$ be a first-order perturbation
of a model edge
metric $g=dr^2\oplus r^2 g^{V}\oplus \pi^*g^B$. Then 
 \[
(2\pi)^k \chi(M) = \int_{M}\Pf^{\tilde{g}}
-\int_B\left( {\Pf(g^B)}\int_{\partialM /B}\oPf(g^{V}) \right)
-\int_{\partialM } \TPf(\nabla^0,\nabla^1).
\]
The  form $\Pf(g^B)$ is zero, by definition, when $\dim{B}$ is odd.
\end{theorem}

Note that the sum of the two boundary terms is itself a
transgression form.

In the particular case when the degeneration is of first order with respect to 
a \emph{conical} metric, we are able to give a geometric expression for the boundary
contribution in the spirit of the classical Gauss-Bonnet formula. Let
\[\mathcal{G}_{j,2k-1}^{\partialM
}:=\frac{1}{j!(2k-1-2j)!}\cB_{g^{\varphi}}\left((R^N)^j\wedge
(\II^g)^{2k-1-2j}\right).
\]
where the second fundamental form $\II^g$ is defined via $\nabla^1$ above.

\begin{theorem} 
Let $g$ be a first-order perturbation of a conical metric $dr^2\oplus
r^2g^N$. Then
\[ (2\pi)^k\chi(M)=\int_{M}\Pf^g-\sum_{j=0}^{k-1}(-1)^j(2j-1)!!\int_{\partialM
}\mathcal{G}_{k-1-j,2k-1}^{\partialM }
\]
\end{theorem}

Similar results, proved with the same techniques, hold for first and second
order perturbations of manifolds with fibered boundary (see Section
\ref{Secaden}). 

The notions of model degenerate metrics studied here, together with their perturbations, 
depend on the choice of a boundary-defining function $x$. A model edge degenerate metric 
with respect to such a function $x$ will look more complicated with respect to a different 
choice $x'$. We refer to the work of Grieser \cite{grieser}, which solves completely 
the conic case, and of Joshi \cite{Jo} dealing with the $b$-case. 
In this work we assume the boundary-defining function $x$ to be fixed once and for all, 
leaving open the quest for an optimal choice of $x$.

\subsection*{Orbifolds}

A natural example
of first-order perturbation of a model edge metric is the complement of a
submanifold $B$ in a Riemannian manifold $M$ when one lifts the original metric
to the oriented blow-up of $B$. The integral of the transgression form from Theorem
\ref{Theorem5} vanishes in this case, reflecting a basic topological fact:
\[\chi(M\setminus B)=\chi(M)-\chi(B).\]
The situation becomes more interesting when we blend in isometric actions 
of finite groups. If $M$ is a Riemannian orbifold with
singularities locally modeled on quotients of type $N/G$ where  $G$ acts freely
on
$N\setminus \Fix_G(N)$ and $\Fix_G(N)$ is a smooth submanifold locus, we
obtain the following Gauss-Bonnet formula for orbifolds:

\begin{theorem}\label{T7} Let $\hat{M}$ be a compact Riemannian orbifold of dimension $2k$
with simple
singularities along $Z\subset \hat{M}$ and let $g$ be the Riemannian metric on
$\hat{M}\setminus Z$. Then 
\begin{equation} \chi(\hat{M})=\frac{1}{(2\pi)^k}\int_{\Int{\hat{M}}}
  \Pf^g +
\sum_{Z_i\in \Fix(\hat{M})}\chi(Z_i) \frac{|G_i|-1}{|G_i|}
\end{equation}
where $\Fix(\hat{M})$ is the set of connected components of the singular locus
of
$\hat{M}$.
\end{theorem}

One should compare Theorem \ref{T7} with the classical Gauss-Bonnet formula for orbifolds of Satake \cite{Sa} (Theorem 2) which expresses the \emph{orbifold Euler characteristic} as an integral of the Pfaffian.

\subsection*{Historical notes} The necessary disclaimer for this subsection is
that our intention is to give 
a slight sense of the huge development of results directly related to
Gauss-Bonnet. 
Voluntary or involuntary omissions  are obviously inevitable.
 
The Gauss-Bonnet formula for polygonal surfaces embedded in Euclidean $3$-space 
was found almost 200 years ago by Gauss, Binet and Bonnet. The standard textbook
formula for closed surfaces in $R^3$ 
linking the Euler characteristic with the integral of the Gaussian curvature
was stated and proved by Walther von Dyck \cite{vD} at the end of the 19th
century.
The modern history of its generalizations can be found in the nice survey 
\cite{Wu}. The integrand in higher dimensions was discovered in the 1920's
by Heinz Hopf in the case of hypersurfaces in Euclidean space, while
the validity of Hopf's formula for embedded manifolds of arbitrary codimension
in $R^n$ was independently proved in 1940 by Allendoerfer and Fenchel, building
on work of Weyl.
In 1943 Allendoerfer and Weil
\cite{AW} not only proved the validity of Hopf's formula 
in the abstract (non-embedded) case, but also gave the correction term for a
manifold
with boundary. They went even further and produced a formula
valid for a topological manifold with boundary which is a Riemannian polyhedron,
i.e., boundary points have neighborhoods which are differentially
modeled on convex cones in $\bR^n$ and there exists a globally defined smooth
Riemannian metric on the resulting differentiable polyhedron. Their theorem is
in some sense at the crossroad of what we call embedded/non-embedded situation.
Soon afterwards, S.\ S.\ Chern \cite{Ch1,Ch2} gave intrinsic proofs for compact
smooth Riemannian
manifolds, both with and without boundary. Chern's articles have been immensely
influential. 
It is worth mentioning here that Chern's theorems, together with Hirzebruch's
signature theorem and the
Hirzebruch-Riemann-Roch formula, constituted the main motivating examples behind
the celebrated 
Atiyah-Singer index theorem.

With regard to more modern
developments, the generalization of the Allendoerfer-Weil theorem of R.\ Walter
\cite{W} on compact locally convex subsets of Riemannian manifolds anticipates
the techniques coming from Geometric Measure Theory with applications to the
integral geometry of subanalytic cycles promoted by J.\ Fu \cite{Fu}. Ideas from
stratified Morse theory have also been used successfully in the context of
integral geometry of tamed sets \cite{BrK}. 
Melrose \cite{Me} proved a Gauss-Bonnet Theorem for $b$-exact metrics as a
corollary 
to his celebrated $b$-Index Theorem.
More recently, an enhanced version
of the Allendoerfer-Weil theorem was used by McMullen \cite{McM} to compute the 
volume of the moduli space of $n$-pointed Riemann surfaces of genus $0$. 
Probabilistic interpretations and proofs of Gauss-Bonnet have been given by
\cite{Ni2}. 
Other important works related to the topic of this paper are cited in the
bibliography. 

\subsubsection*{Acknowledgments:} The first named author had interesting
discussions
about the topic of the article with Jorge de Lira, Luciano Mari and Luquesio
Jorge and for that he would like to thank them. He would like to particularly
thank Vincent Grandjean who patiently listened to the crude ideas that finally
took shape inhere.

\section{The transgressions of the Pfaffian. General facts} \label{S2}
We include here a series of general facts, more or less well-known, about the
transgression of the Pfaffian. There exist various incarnations of the
transgression form (compare for example \cite{Ch2,Gi,Ni,W}) and one of the
purposes of this section is to bring them
under the same umbrella in order to simplify the presentation in the sequel. 
Another purpose is to put together a collection of formulas relating
transgressions 
for different metrics and different metric connections which will
turn out to be essential for our computations in the degenerate metric setting.

\subsection{Transgressions and connections}
Let $E\ra M$ be an oriented Euclidean vector bundle of rank $2k$ over a manifold
$M$. 
Every connection $\nabla$ compatible with the metric
gives rise to a closed form of degree $2k$ on $M$, the
Pfaffian, associated to the curvature tensor $F(\nabla):=d^\nabla\circ \nabla$, 
locally a skew-symmetric matrix of $2$-forms. If 
$F(\nabla)_{ij}:=\langle F(\nabla)s_j,s_i\rangle$ in a local orthonormal basis
$\{s_1,\ldots, s_{2k}\}$ of $E$ then
\begin{equation}\label{Pfeq} \Pf(\nabla):=\frac{1}{2^k k!}\sum_{\sigma\in
S_{2k}}\epsilon(\sigma)F(\nabla)_{\sigma(1)\sigma(2)}\wedge\ldots\wedge
F(\nabla)_{\sigma(2k-1)\sigma(2k)}.
 \end{equation}
In the next section we will define the Pfaffian intrinsically via double forms,
proving its
gauge independence.
What is special about the Pfaffian compared to other invariant
polynomials is that it vanishes in the presence 
of a non-zero parallel section in $E$. 

Given a smooth path of metric connections $\alpha^{\nabla}:=(\nabla^t)_{t\in
[0,1]}$, one can construct a transgression form
$\TPf(\alpha^{\nabla})$ which satisfies
\begin{equation}\label{eq1} d\TPf (\alpha^{\nabla})=\Pf(\nabla^1)-\Pf(\nabla^0).
\end{equation}
The construction goes as follows. On the oriented Euclidean vector bundle
$\pi_2^*E\ra [0,1]\times M$ (where $\pi_2:[0,1]\times M\ra M$ is the projection)
consider the connection $\tilde{\nabla}:=\frac{d}{dt}+\nabla^t$
which acts on a section $(s_t)_{t\in [0,1]}$  of $\pi_2^*E$ as follows:
\[\tilde{\nabla}s_t=dt\otimes\frac{\partial s_t}{\partial t}+(\nabla^ts_t).
\]
Consider the Pfaffian $\Pf(\tilde{\nabla})$ which is a closed form and use the
homotopy formula for $H:=\id_{[0,1]\times M}$ and $\Pf(\tilde{\nabla})$ to
conclude that \eqref{eq1} is valid with
\[ \TPf(\alpha^{\nabla}):=\int_{[0,1]}\Pf(\tilde{\nabla}),
\]
the integral being over the fibers of the projection $\pi_2$.

\begin{example} \label{ex1} Suppose $(M,g)$ is a Riemannian manifold with
boundary of even dimension. Then the Euclidean vector bundle
$TM\bigr|_{\partialM }\ra \partialM $ is endowed with two metric connections.
One is the Levi-Civita connection $\nabla^1:=\nabla^{M}$ on $M$ and the other
one is "the cylindrical connection" $\nabla^0:=d\oplus \nabla^{\partialM }$
where we use the splitting
\begin{equation}\label{eq-1}TM\bigr|_{\partialM }=\bR\nu \oplus T\partialM 
\end{equation}
induced by the unit normal $\nu$. Notice that $\Pf(\nabla^0)=0$ since $\nu$ is a
parallel section and the curvature splits into a direct sum of factors, one of
which is zero. We use the affine path of connections
${\nabla}^s:=(1-s)\nabla^0+s\nabla^1$ to construct the Chern transgression
$\TPf(\nabla^M)=\TPf^g$ associated to the metric $g$. This is the form which
appears in the Gauss-Bonnet formula. 

If the splitting \eqref{eq-1} is extended to a neighborhood $U$ of $\partialM $
(e.g.\ via minus the gradient of the distance function to $\partialM $),
the identity $d\TPf^g=\Pf^g$ on $U$ is valid on $U$.
\end{example}
\begin{rem}\label{rem1} For the reverse path $-\alpha^{\nabla}$ defined
via $-\alpha^{\nabla}(t):=\alpha^{\nabla}(1-t)$ one has:
\[ \TPf(-\alpha^{\nabla})=-\TPf(\alpha^{\nabla}).
\]
Indeed, one uses the orientation-reversing diffeomorphism 
\begin{align*} 
[0,1]\times M\ra [0,1]\times M,&&
(t,m)\ra (1-t,m)
\end{align*}
while fiberwise integration is sensitive to the orientation.
\end{rem}
\begin{prop}\label{P1} For two smooth paths $\alpha^{\nabla}$ and
$\beta^{\nabla}$ of metric connections with
$\alpha^{\nabla}(i)=\beta^{\nabla}(i)$, $i=0,1$ there exists a form
$\TPf(\alpha^{\nabla},\beta^{\nabla})$ of degree $2k-2$ such that:
\begin{equation}\label{eq2}\TPf(\alpha^{\nabla})-
\TPf(\beta^{\nabla})=d\TPf(\alpha^{\nabla},\beta^{\nabla}).
\end{equation}
\end{prop}
\begin{proof} 
Let $\mathcal{A}$ be the space of affine connections compatible
with the metric. It is an affine space modeled on $\Gamma(M;
T^*M\otimes \End^-(E))$.  Let $\Box:=[0,1]\times [0,1]$.
Consider the smooth family of connections
\begin{align*} 
\widetilde{\alpha\beta}:\Box\ra \mathcal{A},&&
\widetilde{\alpha\beta}(s,t)=(1-s)\alpha^{\nabla}(t)+s\beta^{\nabla}(t).
\end{align*}
On the vector bundle $\pi_{2}^*E \ra \Box\times M$ (where
$\pi_2:\Box\times M\ra M$ is the projection) consider the
connection
$\hat{\nabla}:=\frac{d}{ds}+\frac{d}{dt}+\widetilde{\alpha\beta}(s,t)$
which acts on a smooth section $u:\Box\ra \Gamma(M;E)$ of
$\pi_{2}^*E$ via
\[ \hat{\nabla}s=  ds\otimes \frac{\partial u}{\partial
s}+dt\otimes\frac{\partial u}{\partial t}
+\left[(1-s)\alpha^{\nabla}(t)+s\beta^{\nabla}(t)\right](u).
\] 
Applying Stokes formula on $\Box$
to the smooth closed form $\Pf(\hat{\nabla})\in\Lambda^*(\Box\times M)$ we
obtain 
\[ -d\int_{\Box} \Pf(\hat{\nabla})=\int_{\partial \Box}\Pf(\hat{\nabla})
\]
where integration is really integration over the fibers of the projections
$\Box\times M\ra M$ and $(\partial \Box) \times M\ra M$. 
Now $\partial \Box$ consists of two constant paths of connections for $t=0$ and
$t=1$,
while for $s=0$ and $s=1$ by definition the integral on the right hand side
gives the transgressions induces by $\alpha^{\nabla}$ and $\beta^{\nabla}$.
Taking into account the orientations, we get \eqref{eq2} with
\[ \TPf(\alpha^{\nabla},\beta^{\nabla}):=-\int_{\Box}\Pf(\hat{\nabla}).\qedhere
\]
\end{proof}

\begin{notation*} For two metric connections $\nabla^0$ and $\nabla^1$ on $E$ we
denote by  $\TPf(\nabla^0,\nabla^1)$
 the transgression form induced by the affine path $(1-s)\nabla^0+s\nabla^1$.
 \end{notation*}
 
 If $\nabla^0$ is obtained from $\nabla^1$ through a section $s:M\ra
E$ of norm $1$ by using the splitting
 \begin{equation} \label{eq3} E=\bR s\oplus \langle s \rangle^{\perp}
 \end{equation}
 with $\nabla^0:=d\oplus P\nabla^1P$, $P$ being the orthogonal projection on
$\langle s \rangle^{\perp}$ then we set
$\TPf(\nabla^1,s):=\TPf(\nabla^0,\nabla^1)$. We will use the same notation even
if $s$ is only defined along a submanifold $B$ (or boundary) of $M$ with the
understanding that the splitting \eqref{eq3} holds only along $B$, $\nabla^0$ is
a connection on $E\bigr|_{B}\ra B$ and consequently $\TPf$ is a form on $B$.
 
If $s$ is clear from the context, we use $\TPf(\nabla^1)$ for
$\TPf(\nabla^1,s)$. If the connection $\nabla^1$ is the Levi-Civita
connection of a metric $g$, then we use $\TPf^g$ for $\TPf(\nabla^1)$, like 
in Example \ref{ex1}.
 
\begin{prop}\label{p4} For any $4$ metric connections $\nabla^i$, $0\leq i\leq
3$, there
exists a form $\gamma$ such that
\[
\TPf(\nabla^0,\nabla^1)+\TPf(\nabla^1,\nabla^2)+\TPf(\nabla^2,
\nabla^3)+\TPf(\nabla^3,\nabla^0)=d\gamma.
\]
\end{prop}
\begin{proof} Put $\nabla^i$ in cyclic order at the vertices of a smooth
map $\theta:\Box\ra \mathcal{A}$ which on the edges of $\Box$ gives the affine
path connecting $\nabla^i$ and $\nabla^{i+1}$. The proof goes on as in
Proposition \ref{P1}.
\end{proof}

\begin{prop}\label{p5} Let $M$ be a Riemannian manifold (with or without
boundary).
Let $\nabla^0$ and $\nabla^1$ be two metric connections and $s:M\ra E$ a
smooth section of norm $1$. Then there exists a $(2k-2)$-form $\gamma$ such
that the following equality of pairs holds:
\[
(\Pf(\nabla^1),-\TPf(\nabla^1,s))-(\Pf(\nabla^0),-\TPf(\nabla^0,
s))=(-d\TPf(\nabla^1,\nabla^0), \TPf(\nabla^1,\nabla^0)+d\gamma).\]
If $s$ is only defined along a submanifold (or boundary) $B$ of $M$ then the
same relation holds with the second components restricted to $B$.
\end{prop}
\begin{proof} The equality in the first component is clear by \eqref{eq1} and
Remark \ref{rem1}.

For the second component,  let $\nabla^{0c}:=d\oplus
P\nabla^0P$ and $\nabla^{1c}:=d\oplus P\nabla^1 P$, where $P$ is the
projection onto $\langle s\rangle^{\perp}$. Apply Proposition \ref{p4} to
the connections $\nabla^{0c},\nabla^{0},\nabla^{1},\nabla^{1c}$ to
get:
\[
\TPf(\nabla^0,s)-\TPf(\nabla^1,s)+\TPf(\nabla^{1c},\nabla^{0c})=-\TPf(\nabla^0,
\nabla^1)+d\gamma=\TPf(\nabla^1,\nabla^0)+d\gamma.
\]
But $\TPf(\nabla^{1c},\nabla^{0c})=0$ because $s$ is simultaneously
parallel for $\nabla^{0c}$ and
$\nabla^{1c}$ hence $\Pf(\tilde{\nabla})$ vanishes on the affine segment 
of connections from $\nabla^{0c}$ to $\nabla^{1c}$.
\end{proof}
Proposition \ref{p5} has a topological interpretation. Suppose that $s$ is a
unit section of $E\bigr|_{\partialM }$.
Each pair $(\Pf(\nabla^i),-\TPf(\nabla^i,s))$ is closed in
$\Omega^{2k}(M,\partialM ):=\Omega^{2k}(M)\oplus \Omega^{2k-1}(\partialM )$
for the differential
\[ d(\omega,\gamma):=(-d\omega,\iota^*\omega+d\gamma).
\]
Proposition \ref{p5} says that two such pairs determine the same relative
cohomology class. In the compact case, this was proved in \cite{Ci} by showing
that such a pair is Lefschetz dual to the zero locus of a  generic extension of
$s$ to $M$. In the classical case, when $s$ is the unit normal of $\partialM $
this is also a consequence of Chern-Gauss-Bonnet \cite{Ch1} since the map
\[ (\omega,\gamma)\ra \int_{M}\omega+\int_{\partialM }\gamma
\]
gives an isomorphism  $H^{\dim{M}}(M,\partialM )\simeq \bR$.

\begin{prop}\label{p6} Let $(M,g)$ be a manifold, $\pi:E\ra M$ a
Euclidean vector bundle with metric connection $\nabla$, and 
$s_0,s_1:M\ra S(E)$ sections in the sphere bundle of $E$. Suppose there exists a
homotopy $(s_t)_{t\in[0,1]}:M\ra
S(E)$ between the two sections. Then there exists a smooth form $\eta$ such
that:
\[ \TPf(\nabla,s_1)-\TPf(\nabla,s_0)=d\eta.
\] 
\end{prop}
\begin{proof} Let $\tau$ be the tautological section of $\pi^*E\ra S(E)$.
The corresponding "tautological" transgression $\TPf(\pi^*\nabla, \tau)\in
\Omega^*(S(E))$ satisfies:
\begin{align*}
(s_t)^* \TPf(\pi^*\nabla, \tau)={}&\TPf(\nabla,s_t),&(\forall)
~t\in[0,1];\\
d\TPf(\pi^*\nabla,\tau)={}&\pi^*\Pf(\nabla).
\end{align*}
The homotopy formula for the homotopy $H:=(s_t)_{t\in[0,1]}:[0,1]\times M\ra
S(E)$ and $\omega=\TPf(\pi^*\nabla,\tau)$ implies that
\[
\TPf(\nabla,s_1)-\TPf(\nabla,s_0)=d\int_{[0,1]}H^*\omega+\int_{[0,1]}dH^*\omega.
\]
But $dH^*\omega=\pi_2^*\Pf(\nabla)$ where $\pi_2:[0,1]\times M\ra M$ is the
projection. The fiber integral over the fibers of $\pi_2$ of any  form of type
$\pi_2^*\eta$ is zero.
\end{proof}

Proposition \ref{p6} implies the following refinement of Proposition \ref{p5}:
\begin{prop} \label{p8}  Let $M$ be a  manifold with or without boundary, let
$\nabla^0$ and $\nabla^1$ be two metric connections on the Euclidean vector
bundle $E$ and $(s_t)_{t\in[0,1]}:M\ra S(E)$ a smooth homotopy. Then there
exists a $(2k-2)$-form $\gamma$ such that: 
\[
(\Pf(\nabla^1),-\TPf(\nabla^1,s_1))-(\Pf(\nabla^0),-\TPf(\nabla^0,
s_0))=(-d\TPf(\nabla^1,\nabla^0), \TPf(\nabla^1,\nabla^0)+d\gamma).\]
If the homotopy is defined only along a submanifold (or boundary) $B$ then the
second components are defined only over $B$.
\end{prop}

\subsection{Transgressions and metrics} \label{sbs2} 
On an Euclidean vector bundle $V$ of rank $2k$, it is convenient
to identify the space of skew-symmetric endomorphisms  $\End^-(V)$ 
with $\Lambda^2V^*$ by the rule:
\[\End^-(V)\ni A\mapsto a_A(v,w):=\langle v,Aw\rangle=-\langle Av,w\rangle.
\]
Notice that on $\bR^2$, $\begin{bmatrix}0 &1 \\ -1 &0\end{bmatrix}$ goes to
$e_1^*\wedge e_2^*$. The Pfaffian of $A$ is defined by
\[ \Pf(A)=\tfrac{1}{k!}\langle a_A^{\wedge k}, \vol_{V^*}\rangle\in \bR.
\]
In any orthonormal basis of $V$, $\Pf$ is a polynomial with integral
coefficients in the entries of $A$.

Clearly this definition can be extended to endomorphisms $A\in\mathcal{A}\otimes
\End^-(V)$
with values in any algebra $\mathcal{A}$, with the inner product acting
only on the $\Lambda^*V$ component. Then $\Pf(A)\in \mathcal{A}$. In this note,
$\mathcal{A}$ will be the algebra of differential forms on a manifold. 

If $\nabla$ is a metric connection on a Euclidean vector bundle $E$ of rank
$2k$, from the curvature tensor $F(\nabla)\in \Gamma(\Lambda^2T^*M\otimes
\End^-(E))$ we get a form of degree $2$ with values in
$\Lambda^2E^*$ called the
curvature form and denoted here by the same symbol. Explicitly:
\begin{align*}F(\nabla):\Lambda^2TM\otimes \Lambda^2E \ra \bR,&&
F(\nabla)(X,Y;Z,W)=\left\langle
Z,\left([\nabla_X,\nabla_Y]-\nabla_{[X,Y]}\right)W\right\rangle.
\end{align*}
Then $F(\nabla)^{k}\in \Lambda^{2k}T^*M\otimes \Lambda^{2k}E^*$, and
$\Pf(F(\nabla))\in \Omega^{2k}(M)$. This definition agrees with 
\eqref{Pfeq}. The operation of contraction with the volume element in the second
component is sometimes called Berezin integral. Double forms, i.e., sections of
$\Lambda^*T^*M\otimes \Lambda^*E^*$, form an algebra.

From now on we take $E=TM$.
Let $g_0, g_1$ be two Riemannian metrics on $M$, and
$\nabla^{g_0},\nabla^{g_1}$ 
the corresponding Levi-Civita connections. We want to find an explicit primitive
of the 
difference $\Pf(R^{g_1})-\Pf(R^{g_0})$. Set $g_s=(1-s)g+s g_1$, a $1$-parameter
family
of  Riemannian metrics on $M$, and define 
a Riemannian metric on $X:=[0,1]\times M$ as a generalized cylinder \cite{bgm}:
\[
G=ds^2+g_s.
\]
It is easy to see that for every $x\in M$,
the intervals $[0,1]\times\{x\}$ are geodesics in $X$. Therefore, parallel
transport on $X$ 
along these intervals preserves the orthogonal complement to $\partial_s$, i.e.,
$TM$. We get  for each $s$ a vector bundle isometry
\[
\tau_s:(TM,g_0)\to (TM,g_s).
\]
We identify in this way for all $s$ the Euclidean vector bundles with metric
connections
$(TM,g_s,\nabla^{g_s})$ with $(TM,g_0,\nabla^s)$, 
where $\nabla^s=\tau_s^{-1}\nabla^{g_s}\tau_s$. 
Clearly such an identification preserves the Pfaffian of the curvature:
\[
\Pf(R^{g_s})=\Pf(R^s),
\]
where $R^s=F(\nabla^s)$ is the curvature of $\nabla^s$.
Write 
\[\Pf(R^{g_1})=\Pf(R^{g_0})+\int_0^1 \frac{d}{d s} \Pf(R^{g_s})ds
=\Pf(R^{g_0})+\int_0^1 \frac{d}{d s} \Pf(R^{s})ds.\]
The advantage of the second expression over the first is that now we work in a
fixed Euclidean vector bundle $(TM,g_0)$ endowed with a family in $s$ 
of metric connections 
$\nabla^s$, and the coefficients of the Pfaffian polynomial depend on the
metric but not on the connection. We compute
\[{\partial_s} \Pf(R^{s})\otimes\vol_{g_0}=
\frac{1}{k!}{\partial_s}\left((R^s)^k\right)=
\frac{1}{(k-1)!}\dot{R}^s\wedge
\left((R^s)^{k-1}\right).
\]
It is well-known that $\dot{R}^s$ is $d^{\nabla^s}$-exact: indeed, 
let $u,v$ be vector fields on $X$ tangent to
$M$ and parallel in the $\partial_s$ direction. For every vector field $Y$ on
$M$ 
constant in $s$ (i.e., $[\partial_s,Y]=0$), write
\[\langle \nabla_Y^s u,v\rangle = \langle \nabla^0_Y u,v\rangle+ \langle
\theta^s(Y) u,v\rangle.
\]
Then $\dot\nabla^s=\dot\theta^s$ and so $\dot{R}^s= d^{\nabla^s}\dot{\theta}^s$.
From the second Bianchi identity, $d^{\nabla^s}R^s=0$, so 
\[
\dot{R}^s\wedge
(R^s)^{k-1}=d^{\nabla^s}\left(\dot\theta\wedge(R^s)^{k-1}\right).
\]
For every double form $\mu\in\Lambda^*M\otimes\Lambda^{2k}M$, 
write $\mu=\cB_{g_0}\mu\otimes \vol_{g_0}$, where
$\cB_{g_0}$ is the Berezin integral with respect to $g_0$. 
Since $\vol_{g_0}$ is parallel, we have
$d^{\nabla^s}\mu = d(\cB_{g_0}\mu)\otimes\vol_{g_0}$. Hence
\[\frac{\partial}{\partial s} \Pf(R^{s})=\frac{1}{(k-1)!}
d\left(\cB_{g_0}\left(\dot\theta^s\wedge(R^s)^{k-1}\right)\right).\]
It follows that
\begin{align}\label{tfp}
\Pf(R^{g_1})=\Pf(R^{g_0})+\frac{1}{(k-1)!}d\left(\int_0^1
\cB_{g_0}(\dot\theta^s\wedge(R^s)^{k-1})\right).
\end{align}

\begin{prop} \label{P5} Let $\alpha^{\nabla}(s):=\nabla^s$ be the above family
of
$g_0$-compatible connections. Then
\[\frac{1}{(k-1)!}\int_0^1
\cB_{g_0}(\dot\theta^s\wedge(R^s)^{k-1})= \TPf(\alpha^{\nabla}).\]
\end{prop}
\begin{proof} Let $\tilde{\nabla}:=\frac{d}{ds}+\nabla^s$ be the connection on
$\pi_2^*TM$ used in the previous subsection. By definition
$\TPf(\alpha^{\nabla})=\int_{[0,1]}\Pf(\tilde{\nabla})$, where the  integration
is over the fibers of $\pi_2:X\ra M$. 

Every form $\gamma$ on $X$ is a sum of type $ds\wedge \omega_s+\eta_s$ where
$(\omega_s)_{s\in [0,1]}$ and $(\eta_s)_{s\in[0,1]}$ are smooth families of
smooth forms on $M$. Fiber integration kills the component $\eta_s$ which does
not contain the volume form of the fiber. In other words:
\[ \int_{[0,1]}\gamma=\int_{[0,1]} ds\wedge \omega_s=
\int_0^1\omega_s~ds=\int_0^1\iota_{\partial_s}\gamma~ds.
\]
The integrals $\int_0^1(\cdot)~ ds$ is to be understood as integrals of
functions (of $s$) with values in $\Lambda^*T_pM$ for  $p\in M$.  
We need to compute $\iota_{\partial_s}(\Pf(\tilde{\nabla}))$. First notice that
$ F(\tilde{\nabla})=ds\wedge \dot{\nabla^s}+F(\nabla^s).$
Then
\begin{equation}\label{F1} F(\tilde{\nabla})^k=k\; ds\wedge \dot{\nabla}^s\wedge
F(\nabla^s)^{k-1}+F(\nabla^s)^k.
\end{equation}
The contraction operation $\iota_{(\cdot)}$ can be defined equally well on forms
with values in an algebra. Then 
\[
\iota_{\partial_s}(\Pf(\tilde{\nabla}))=\frac{1}{k!}\iota_{\partial_s}(\cB_{g_0}
(F(\tilde{\nabla})^k))=\frac{1}{k!}\cB_{g_0}(\iota_{\partial_s}[F(\tilde{\nabla}
)^k])\]
where  $\iota_{\partial_s}$ acts by definition only on the first component of a
double form\footnote{The curvature form is in general a section of
$\Lambda^2T^*M\otimes \Lambda^2E^*$.}. The second equality holds because
$\cB_{g_0}$ acts on the second component only of the double form. 
By \eqref{F1},
\[\iota_{\partial_s}(\Pf(\tilde{\nabla}))=\frac{1}{(k-1)!}\;\cB_{g_0}(\dot{
\theta}
^s\wedge F(\nabla^s)^{k-1}).\qedhere
\]
\end{proof}

\section{Gauss-Bonnet on manifolds with boundary}
This section contains a  proof of the well-known version of
Gauss-Bonnet on manifolds with boundary proved by Allendoerfer-Weil  and Chern
$80$ years ago.  While both versions of Gauss-Bonnet have received many proofs, 
the main idea we use in this section seems natural and delivers a direct proof. 
It will appear again in the degenerate metric case.   

 Briefly, the generalization of Gauss-Bonnet to manifolds with boundary where
the metric is of product type near the boundary is a triviality and of course
there is no contribution from the boundary. In order  to find the  "defect" in
the non-product metric case we use parallel transport to produce \emph{tangent
bundle isometries} between a product metric and the one we are interested in
(origins of this idea can be traced to \cite{GaBo,Gi}). Then we use properties
of  transgressions. The boundary integrand we obtain is not obviously equal to
the standard one obtained by Chern \cite{Ch2} and we clarify at the end of the
section why this is the case. We use the formalism of double forms which
simplifies the presentation to a certain extent.

Let $g$ be a smooth metric on a compact manifold $M^{2k}$ with boundary
$\pM$.  Let $R^h\in\Lambda^2 \pM\otimes\Lambda^2\pM$ be the curvature form
of the
boundary with  respect to the induced metric $h$ and $\II$ the second
fundamental form of 
$\pM\hookrightarrow M$. Our convention here is the following:
\[\II(X,Y)=-\langle\nabla_X\nu,Y\rangle
\]
where $\nu$ is the exterior unit normal. We will use the symbol $\II$ also for
the $(1,1)$ double form on $\partialM $ determined by $\II$. 
We denote by $\Pf^g$ the Pfaffian of $g$ and by $\TPf^g$ the transgression form
on
$\partialM $ constructed from $\nabla^g$ and $d\oplus \nabla^h$ (see Example
\ref{ex1}) where $\nabla^g$ and $\nabla^h$ are the Levi-Civita connections on
$M$ and $\partialM $ respectively. We give a direct proof of the 
Allendoerfer-Weil-Gauss-Bonnet-Chern \cite{Ch2} formula for manifolds with
boundary
using the formalism of double forms.

\begin{proof}[Proof of the Gauss-Bonnet-Chern formula \eqref{eq-4}]
Let $g_1:=g$.
Using the unit geodesic flow normal to the boundary, we can write $(M,g)$
as a generalized cylinder \cite{bgm} near the boundary:
\[
g=dt^2+h(t),
\]
where $h(t)$ is a smooth family of symmetric $2$-tensors on $\pM$, and $h(0)$ is
a metric.
Take $g_0$ to be any metric which in the same product decomposition near the
boundary looks like
\[
g_0=dt^2+h(0),
\]
i.e., $g_0$ is of product type near the boundary and induces the same metric 
$h(0)$ on $\pM$ as $g$. By the Gauss-Bonnet formula for product-type metrics
(obtained by doubling the manifold for example) 
and the transgression formula \eqref{tfp}, we get
\begin{equation}\label{et}
(2\pi)^k\chi(M)=\int_M\Pf(R^{g_0})=\int_M\Pf(R^{g})
-\frac{1}{(k-1)!}\int_0^1\int_{\pM}
\cB_{g_0}\left(\dot\theta^s\wedge(R^s)^{k-1})\right).
\end{equation}

Notice  that all metrics $g_s$ coincide on $TM\bigr|_{\partialM }$. One
consequence is that all bundle isometries $\tau_s$ when restricted to
$TM\bigr|_{\partialM }$ are equal to the identity. Hence every Levi-Civita
connection $\nabla^{g_s}$ when restricted to $TM\bigr|_{\partialM }$ is equal to
$\nabla^s$ and all are metric compatible whether we refer to $g_0$ or $g$. By
Proposition \ref{P5} the integral on the boundary in \eqref{et} is in fact equal
to $\iota^*\TPf(\alpha^{\nabla})$ where $\iota^*:\partialM \ra M$ is the
inclusion and $\alpha^{\nabla}(s)=\nabla^s$. By Proposition \ref{P1} when
integrating over the boundary, it does not matter what path of connections one
takes between the first and the last connection so we might as well take the
segment.  
To complete the proof of \eqref{eq-4} we still have to identify explicitly the
transgression term from
\eqref{et}. 

First, the Berezin integrals with respect to $g$  and to $h$ at the boundary 
are related by 
\[
\cB_g(dt\wedge \mu)=\cB_h(\mu)
\]
for every form $\mu\in\Lambda^{2k-1}\pM$.

The difference $\nabla^{g_s}-\nabla^{g_0}$ is a $\End^-(TM)$-valued $1$-form.
Define $\theta^s\in\Lambda^1(\partialM )\otimes \End^-(TM\bigr|_{\partial
M})$ as the pull-back of this $1$-form to the boundary.
We claim that $\theta^s$, viewed as a $(1,2)$ double form, equals
\begin{equation}\label{eq11}
\theta^s=(1\otimes dt)\wedge s\II^g.
\end{equation}
Indeed, notice that
$\langle\nabla^{g_s}_XY,Z\rangle=\langle\nabla^{g}_XY,Z\rangle$ for all
$X,Y,Z\in T\partialM $ as $g_s\equiv h$ on $T\partialM $. Moreover
$\langle\nabla^{g_s}_X\partial_t,\partial_t\rangle=0$ for all $s$ and $X\in
T\partialM $. Hence with respect to the decomposition $TM\bigr|_{\partial
M}=\bR\partial_t\oplus T\partialM $ and the corresponding decomposition of
$\End^-(TM\bigr|_{\partialM })$, the only non-zero components of  $\theta^s$ are
off-diagonal. Then for $X,Y\in T\partialM $
\begin{equation}\label{eq10}\begin{split}
\langle\theta_{X}^s(Y),\partial_t\rangle={}&\II^{g_s}(X,
Y)= -\langle \nabla^{g_s}_X
\partial_t,Y\rangle=-\frac{1}{2}(L_{\partial_t}g_s)(X,Y)\\
={}&-\frac{s}{2}h'(0)(X,Y)=
-\frac{s}{2}L_{
\partial_t}g(X,Y)= s\II^g(X,Y).
\end{split}\end{equation}
where we used Lemma \ref{L1} in the first line. Notice that \eqref{eq10} is a
rewriting of \eqref{eq11}.

Since $\nabla^s=\nabla^0+\theta^s$ we get that
$R^s=R^0+d^{\nabla^0}\theta^s+\theta^s\circ \theta^s$ where we use the symbol
$\circ$ instead of the more popular $\wedge$ in order to distinguish it from
the product for double forms.

 On one hand, $R^0=0\oplus R^h$ with respect to $TM\bigr|_{\partial
M}=\bR\partial_t\oplus T\partialM $. Hence as $(2,2)$ forms on $\partialM $ one
has $R^0=R^h$. Second, $d^{\nabla^0}$ also respects this decomposition so
$d^{\nabla^0}\theta^s$ will be a $2$-form with non-zero values only on  the
anti-diagonal blocks of $\End^-(TM\bigr|_{\partialM })$. It follows that, when
writing $d^{\nabla^0}\theta^s$ as a double form, the second component will
always contain a $dt$. But $\dot{\theta}^s$ also contains a $dt$ in its second
component. So in $\dot{\theta}^s\wedge (R^s)^{k-1}$ this product vanishes.

We are left with turning $\theta^s\circ \theta^s$ into a double form.  If
$\{\partial_t,e_2,\ldots,e_n\}$ is an oriented orthonormal basis for $TM$ at a
point $p\in \partialM $ then at $p$, $\theta^s$ is a skew-symmetric matrix  with
non-zero terms only along the first line and the first column. In fact
$\theta^s_{1i}=s\II^g(\cdot,e_i)$, $i\geq 2$ and
 \begin{align*}
 (\theta^s\circ\theta^s)_{ij}=-s^2\II^g(\cdot,e_i)\wedge\II^g(\cdot,e_j),&& i<j.
\end{align*}
This represents the $(2,2)$ double form 
\[-s^2\sum_{2\leq i<j}\II^g(\cdot,e_i)\wedge\II^g(\cdot,e_j)\otimes e_i^*\wedge
e_j^*.\]
 On the other hand
\[\II^g\wedge\II^g=\left(\sum_{i\geq 2}\II^g(\cdot,e_i)\otimes
e_i^*\right)\wedge\left(\sum_{i\geq 2}\II^g(\cdot,e_i)\otimes
e_i^*\right)=2\sum_{2\leq i<j}\II^g(\cdot,e_i)\wedge\II^g(\cdot,e_j)\otimes
e_i^*\wedge e_j^*.
\]
Hence $\theta^s\circ \theta^s=-\frac{s^2}{2}\II^g\wedge\II^g$, and so
the integrand over $\partialM $ in \eqref{et} is
\begin{align*}
\lefteqn{\frac{1}{(k-1)!}\int_0^1\cB_g\left((1\otimes dt)\wedge \II^g\wedge 
\left(R^h-\frac{s^2}{2}(\II^g)^2\right)^{k-1}\right)~ds=}\\
{}&=\frac{1}{(k-1)!}
\cB_h\left(\sum_{j=0}^{k-1}{\binom{k-1}{j}}
\frac{(-1)^j}{2^j}\frac{1}{2j+1}({\II^g})^{2j+1}\wedge
(R^h)^{k-1-j}\right).\qedhere
\end{align*}
\end{proof}

The next simple Lemma is quite well-known, and will be widely used in this
article.
\begin{lem}\label{L1} Let $TM$ be endowed with a metric $G$ and corresponding
Levi-Civita connection $\nabla$. Let $X\in \Gamma(TM)$ be a vector field such
that $X^{\sharp}$ is a closed $1$-form (e.g. if $X$ is gradient). Then
\[ G(\nabla_YX,Z)=\tfrac{1}{2}(L_XG)(Y,Z).
\] 
\end{lem}
\begin{proof} 
Directly from the Koszul formula one has
\[2G(\nabla_YX,Z)=(L_XG)(Y,Z)+dX^\sharp(Y,Z).
\]
By hypothesis the second term vanishes.
\end{proof}
\begin{rem}\label{s1fr} Not only  that the integral over $\partialM $ of
$\TPf^g$ equals the integral on $\partialM $ of the right hand side of
\eqref{et} but the integrands themselves coincide. This is because the
Levi-Civita connection for $g_s=(1-s)g_0+sg$, when restricted to
$TM\bigr|_{\partialM }$ coincides with $(1-s)\nabla^{g_0}+s\nabla^{g}$. This
follows from  $g_s\equiv g_0$ on $TM\bigr|_{\partialM }$ for all $s$ and from
the Koszul formula which always gives:
\[ \langle \nabla^{g_s}_XY,Z \rangle_{g_s}=(1-s)\langle
\nabla^{g_0}_XY,Z\rangle_{g_0}+s\langle \nabla^{g_1}_XY,Z \rangle_{g}.
\]
\end{rem}
\begin{rem}  Let 
\[\frac{\cB_h\left((R^h)^j\wedge \II^{2k-1-2j}\right)}{j!
(2k-1-2j)!}=:\mathcal{G}_{j,2k-1}^h.\]
 Then the integral of the transgression form has the following {\ae}sthetically
pleasing form
 \[\sum_{j=0}^{k-1}(-1)^{j}(2j-1)!!\int_{\partialM }\mathcal{G}_{k-1-j,2k-1}^h.
 \]
\end{rem}

\begin{example} \label{Exemp1} The Gauss-Bonnet formula \ref{eq-4} applied to
the
unit disk $D^{2n}\subset \bR^{2n}$ anticipates that
\[\frac{1}{(2\pi)^n} \int_{S^{2n-1}}\TPf^g=-1.
\]
The sphere is oriented with the outer normal first convention. We compute the
right hand side of \eqref{eq-4} to check this. On one hand, $\II=-h$, where $h$
is the round metric. On the other hand, Gauss equation gives
$0=R^h-\frac{1}{2}\II\wedge\II$, hence 
\[\cB_h((R^h)^{j}\wedge
\II^{2k-2j-1})=-\frac{1}{2^j}\cB_h(h^{2k-1})=-\frac{1}{2^j}(2k-1)!\vol_{h}.\]
Using that $\vol{(S^{2k-1})}=\frac{2\pi^k}{(k-1)!}$ we get
 \begin{align*} 
 \lefteqn{\sum_{j=0}^{k-1}c(j,k)\int_{S^{2n-1}}\cB_h((R^h)^{j}\wedge
\II^{2k-2j-1})}
\\ {}&=-\sum_{j=0}^{k-1}\frac{(-1)^{k-1-j}}{
2^{k-1-j}}\frac{1}{j!}\frac{1}{(k-1-j)!}\frac{1}{2k-2j-1}\frac{(2k-1)!}{2^j}\vol
{(S^{2k-1})}\\ 
{}&=-\frac{(2k-1)!}{2^{k-1}}\frac{2\pi^k}{[(k-1)!]^2}\sum_{j=0}^{k-1}
\frac{(-1)^j\binom{k-1}j}{2j+1}.
\end{align*}
 Notice  that
  \[\sum_{j=0}^{k-1}\frac{(-1)^j{\tbinom{k-1}j}}{2j+1}
=\int_0^1(1-x^2)^{k-1}~dx=\int_0^{\pi/2}(\cos\theta)^{2k-1}
~d\theta=\frac{2^{2k-2}[(k-1)!]^2}{(2k-1)!}.\]
  Hence
  \[\frac{1}{(2\pi)^k}\sum_{j=0}^{k-1}c(j,k)\int_{S^{2n-1}}\cB_h((R^h)^{j}\wedge
\II^{2k-2j-1})=-1.
  \]
\end{example}
\begin{rem} The integrand in \eqref{eq-4} on $\partialM $ coincides with Chern's
integrand \cite{Ch1}. Chern's transgression, which lives on the spherical bundle
$SM$, can be written (see for example \cite{W}) as\footnote{The negative sign in
front of the sum
is there so that $d\Pi=\pi^*\Pf^g$.}
\begin{align}\label{eqAi} 
\Pi:= -\sum_{j=0}^{k-1}a_iA_i,&&a_i=[(2\pi)^{k}i!(2k-2i-1){!}{!}]^{-1},&&
A_i=(\pi^*\mathcal{R})^i\wedge I\wedge
(DI)^{2k-2i-1}.
\end{align}
In \eqref{eqAi}, $\mathcal{R}$ is the curvature form on $M$,  $I:SM\ra
\pi^*TM$ is the tautological section seen as a $0$-form on $SM$ with values in
$\pi^*TM$ and $DI=(\pi^*\nabla)I$ is the covariant derivative seen as a $1$-form
with values in $\pi^*TM$. Hence one works in the algebra of forms on $SM$ with
values in $\Lambda^*\pi^*TM$. Wedging
with $I$ kills the normal component in any product
$DI^{2k-2h-1}$ and also in $(\pi^*\mathcal{R})^i$.

Given a hypersurface $N$ oriented by the normal $\nu$ one has that
$\nu^*(I\wedge DI)=\nu\wedge\nu^*(DI)$ actually equals $-\nu \wedge\II_N$
where
$\II_N:TN\ra TN$ is the second fundamental form seen as the endomorphism
$-\nabla \nu$.
Moreover $\nu^*\mathcal{R}$ is the tangential component of the curvature tensor
of $M$ restricted to $N$. Let $\II:=\II_N$ and $R^N$ the curvature form on
$N$.  Gauss Equation gives
 \[ \nu^*{\mathcal {R}}=R^N-\frac{1}{2}\II\wedge\II.
 \]
 Therefore
 \[-\nu^*(A_i)=(R^N-1/2{\II}\wedge {\rm II})^i\wedge\nu \wedge
\II^{2k-2i-1}
 \]
 and we must check that
 \begin{align*} 
 \nu^*\Pi= {}&\sum_{i=0}^{k-1}\sum_{j=0}^i\frac{1}{i! \cdot 1\cdot 3 \ldots
\cdot
(2k-2i-1)}\frac{(-1)^j{\binom{i}{j}}}{2^j}
\II^{2k-2(i-j)-1}(R^N)^{i-j}\\
={}&\sum_{j=0}^{k-1}\frac{{\binom{k-1}{j}}}{(k-1)!}
\frac{1}{2j+1}\frac{(-1)^j}{2^j}\II^{2j+1}(R^N)^{k-1-j}=:\TPf^N,
 \end{align*}
This equality follows from the elementary identity of double factorials
\[\sum_{j=0}^p(-1)^j\frac{(2p)!!}{(2j)!!(2p-2j+1)!!}=\frac{(-1)^p}{2p+1}.
\]
\end{rem}

\section{Conical manifolds}\label{conman}

Let $N$ be a compact oriented manifold, possibly disconnected. A \emph{conical
singularity} modeled on $N$ is a
Riemannian 
metric on $(-\epsilon,0)\times N$ of the form
\begin{equation}\label{modelcone}
g_c=dr^2\oplus f^2(r)\cdot h(r)
\end{equation}
where $h(r)$ is a smooth family of Riemannian metrics on $N$ down to $r=0$, and
$f:(-\epsilon,0]\ra [0,\infty)$ is a function with the following properties
\begin{itemize}
\item[(i)] $f$ is smooth on $(-\epsilon,0)$;
\item[(ii)] $f$ vanishes only at $0$;
\item[(iii)] $f$ is $C^1$ at $0$. 
\end{itemize} 
Notice that, as a consequence of the hypotheses, $f'(0)\leq 0$. 
\begin{definition}
When $h(r)\equiv h$ is constant and $f(r)=-\theta r$ with $\theta>0$ we call the
conical singularity a \emph{geometric cone} of inclination $\theta$.
\end{definition}
The smoothness at $r=0$ of $h(r)$ needs to be emphasized. 
There are two equivalent formulations for this property:
\begin{enumerate}
\item The metric $dr^2\oplus h(r)$ is the restriction to $(-\epsilon,0)\times N$
of a
smooth metric
on $(-\epsilon,\epsilon)\times N$;
\item The family $(-\epsilon,0)\ni r\mapsto h(r)\in C^\infty(N,T^*N\otimes
T^*N)$ 
has a limit at $r=0$ together with all its derivatives in $r$.
\end{enumerate}
\begin{definition} \label{dvp} An oriented manifold with conical-type
singularities is a Riemannian manifold $(M,g)$ such that there exists a compact
set $K$ and an orientation preserving diffeomorphism $\varphi: M\setminus
K\simeq (-\epsilon,0)\times N$ such that on $M\setminus K$:
\[ g=\varphi^*g_c.
\]
\end{definition}

We now define some polynomials in the curvature of a Riemannian manifold
$(N,h)$ of dimension $n$ using the Berezin integral $\cB_h$ where
 \[h:=h(0).\]

\begin{definition}\label{LKcu} The Lipschitz-Killing curvature (see \cite{La} or
\cite{Mo}) of level $j$ is, up to a normalization constant, the following
form of degree $n$ on $N$:
\[ P_{j,n}(h)=\frac{1}{j!(n-2j)!}\cB_h\left((R^h)^{j}\wedge h^{n-2j}\right).
\]
\end{definition}
Like the Pfaffian, in any orthonormal base the form $P_{j,n}$ is a polynomial
with 
integral coefficients in the components of $R^h$. The Lipschitz-Killing
curvatures 
are familiar objects and they appear in Weyl's tube formula. 

\begin{example} Here are a few examples of Lipschitz-Killing curvatures:
 \begin{align*}P_{0,n}(h)=\vol_h,&& P_{1,n}(h)=\frac{1}{2}\scal_h\cdot\vol_h,&&
P_{k,2k}=\Pf(R^h).
\end{align*}
\end{example}
\begin{rem}  Let $\tilde{N}:=(-\epsilon,0)\times N$ be a  geometric cone of
inclination $c>0$. Then the transgression form for each slice $\{r\}\times N$
does not depend on $r$. Indeed the Levi-Civita connection and the cylindrical
connection obtained from it are the same for $T\tilde{N}\bigr|_{\{r\}\times N}$
irrespective of $r$. Denote this transgression form by $\TPf(N,h,c)$. {For the 
inclination $c=0$, set $\TPf(N,h,c)=0$.}
\end{rem}
We prove now the main result of this section.

\begin{theorem} \label{Tconca}
Let $(M^{2k},g)$ be an oriented manifold with conical-type singularities
modeled on a possibly disconnected manifold $N$ with induced metric $h$. 
Then
\begin{align}\label{GBcs1}
(2\pi)^k \chi(M) =& \int_{M}\Pf^g-\int_{N}\TPf(N,h,-f'(0))\\\label{GBcs2}
 =& \int_M\Pf^g 
+ \sum_{j=0}^{k-1}[f'(0)]^{2k-2j-1}\tilde{c}(k-1-j)\int_N
P_{j,2k-1}(h)
\end{align}
with
\begin{equation}
\label{tildec}\tilde{c}(l)=(-1)^{l}\cdot(2l-1)!!.
\end{equation}
\end{theorem}
\begin{proof} For each $r\in(-\epsilon,0)$, let $M_r$ be the complement of
$\varphi^{-1}((r,0)\times N)$. It is a compact manifold with boundary and
therefore \eqref{eq-4} applies to it:
\[ (2\pi)^k\chi(M_r)=\int_{M_r}\Pf^g-\int_{\partialM _r}\TPf^g.
\]
Clearly all $M_r$ are homotopic to each other so the left hand side does not
change with $r$. We will show  that 
\[ \lim_{r\ra 0}\int_{\partial
M_r}\TPf^g=-\sum_{j=0}^{k-1}[f'(0)]^{2k-1-2j}\tilde{c}(k-1-j)\int_{N}P_{j,2k-1}
(h).\]
This will also prove the convergence of $\int_{M_r}\Pf^g$ when $r\ra 0$. 
(In Section \ref{Sec8} we prove the stronger statement that
$\Pf^g$ is a smooth form on $M$, including at the boundary.)

The first observation is that the Levi-Civita connection $\partialM _r$ with the
metric $h_1(r):=f(r)^2h(r)$ is the same as the Levi-Civita connection for the
metric $h(r)$, hence as operators 
\[ R^{h_1(r)}={f^2}(r)R^{h(r)}\]
due to the metric dependence of the identification $\End^-(V)\simeq \Lambda^2
V^*$.

One is left computing the evolution of $\II^r$ for $\partialM _r$. Since
$\partial_r$ is a gradient vector field we apply Lemma \ref{L1} again:
\[ \II^r(X,Y)=-\langle \nabla^g_X\partial_r,Y \rangle=-\frac{1}{2}L_{\partial
r}(dr^2+f^2(r)h(r))=-[f'(r)f(r)h(r)+\frac{f^2(r)}{2}h'(r)].
\]
We also have $\cB_{h_1(r)}=\left(f(r)\right)^{1-2k}\cB_{h(r)} $ and so
\[ \cB_{h_1(r)}\left((R^{h_1(r)})^j\wedge
(\II^r)^{2k-1-2j}\right)=-f'(r)^{2k-1-2j}\cB_{h(r)}\left((R^{h(r)})^{j}\wedge
h(r)^{2k-1-2j}\right)+o(f(r)).
\]
Multiply this with $c(j,k)=\frac{\tilde{c}(k-1-j)}{j!(2k-1-2j)!}$, 
take the sum in $j$ and the limit $r\ra 0$ to get \eqref{GBcs2}.

To see that  \eqref{GBcs1} is true, recall (for example Remark \ref{s1fr}) that
$\int_{N}\TPf(N,h,c)$ can be computed also as a sum of integrals over $N$ of
products $\II^{2k-1-2j}\wedge R^j$ where $\II$ and $R$ are the second
fundamental form respectively the curvature form of a slice of a geometric
cone. But for such a geometric cone, $\II$ is a multiple of the metric,
and the computations go as before. 
\end{proof}

We notice thus that for an odd-dimensional manifold the total Lipschitz-Killing
curvatures can be recovered as coefficients of the integral of a certain
transgression. We state this separately.

\begin{cor}\label{cor1} For a geometric cone modeled  on $(N,h)$   of
inclination $\theta$ with $\dim{N}=n$,  $n$ odd, the following holds:
\begin{align*}
\int_{N}\TPf(N,h,\theta)=\sum_{j=0}^{(n-1)/2}\theta^{n-2j}\tilde{c}\left(
\tfrac{n-1}{2}-j\right)\int_{N}P_{j,n}(h).
\end{align*}
\end{cor}
\begin{proof} The function is in this case $f(r)=-\theta r$.
\end{proof}
\begin{rem} \label{nrem} The odd Pfaffian from the Introduction can in fact be seen
as a transgression:
\[ \oPf(h)=\TPf(N,h,1).
\]
Notice that in the case when $N=S^{2n-1}$ with the
round metric we get:
\[ \int_{S^{2n-1}}\TPf(S^{2n-1},\round,1)=(2\pi)^n.
\]
One can compare this with Example \ref{Exemp1}. The difference in sign has to
do with the fact that  $S^{2n-1}$ seen as a geometric cone is oriented with the
inner normal first since that is the direction of $\partial_r$ that points
towards the "singularity".
\end{rem}

\begin{rem} We can construct a manifold with boundary $\tilde{M}:=M\cup
(-\epsilon,0]\times N/\sim$ where the identification is made via the
diffeomorphism $\varphi$ of Definition \ref{dvp} in an obvious way. The
degenerated conical metric $g$ induces a pseudo-distance on $\tilde{M}$ in which
the  (pseudo) distance between any two points on $\partial\! \tilde{M}$ is
zero. 
Collapsing the boundary of $\tilde{M}$ to a point gives a metric space $\hat{M}$
which is homeomorphic to the one point compactification of $M$. Then
\[ \chi(\hat{M})=1+\chi(M).
\]
If the singular space $\hat{M}$ is the focus of the analysis, then we can say
that the singularity, or the point at $\infty$ contributes to the Euler
characteristic with  the quantity
\[
1+\frac{1}{(2\pi)^k}\sum_{j=0}^{k-1}f'(0)^{2k-1-2j}\tilde{c}(k-1-j)\int_NP_{j,
2k-1}(h).
\]
\end{rem}
\begin{example}
In the case $k=1$  the contribution of the singularity is (recall that
$f'(0)\leq 0$, $\tilde{c}(0,1)=1$)
\[ 1+ \frac{f'(0)}{2\pi}\length_h(N).
\]
This fits with two opposite examples. The first is a closed surface $S$
embedded in $\bR^3$ with a cuspidal singularity. Then $f'(0)=0$. The
geometric contribution to the Euler characteristic of  the cusp is $1$ which is
the area of the half unit sphere divided by $2\pi$. The half unit sphere is the
normal cycle of the cusp, or the solid angle described by the variation of a
unit normal to each surface of a family of smooth surfaces contained in the
bounded region of $S$ and converging to $S$. 

The other example is when $N=S^1$ with the round metric and $f'(0)=-1$. Then
$\hat{M}$ is a closed surface with smooth metric (see \cite{Pe} p.\ 13, Prop.1)
and the
contribution of the removable singularity vanishes, recovering Gauss-Bonnet for
$\hat{M}$ in this case.
\end{example}

\section{Edge manifolds: the model metrics} \label{Emmm}
Let $N$ be an $n$-dimensional closed, oriented manifold. Assume $\pi:N\ra B$ is
a locally trivial fiber bundle with vertical bundle $\VN $ and suppose $\pi$ is
endowed with an Ehresmann connection $\mathcal{E}\in \Hom(TN,\VN  )$ that
induces a
decomposition
\[ TN=\VN  \oplus \pi^*TB.
\]
An (incomplete) edge singularity modeled on $(N,\pi,\mathcal{E})$ is a metric on
$(-\epsilon,0)\times N$ of the type $dr^2\oplus r^2g^V\oplus \pi^*g^B$
where $g^V$ and $g^B$ are metrics on $\VN  $ and $TB$ respectively. More
generally, a \emph{model edge metric} will be any metric of type:
\[ g_e=dr^2\oplus r^2g^{V}(r)\oplus \pi^*g^B
\]
where $g^{V}(r)$ is a smooth family of metrics down to $r=0$.
We set:
\begin{align}\label{gN} g^{V}:=g^{V}(0),&& g^N:=g^{V}\oplus \pi^*g^B.
\end{align}
The Levi-Civita connection $\nabla^N$ of the metric $g^N$ on $N$ induces a
connection $\nabla^{\VN  }$ on $\VN  $ via $\mathcal{E}\nabla^N \mathcal{E}$. We
will call it the orthogonal projection.
Clearly $\nabla^{\VN  }$ restricted to each fiber $N_b$ is the Levi-Civita
connection of that fiber for the metric $g^V$.

\begin{definition} A manifold with edge singularities is a smooth manifold $M$
with a Riemannian metric $g$ such that there exists a compact set $K$ and a
diffeomorphism $\varphi:M\setminus K\ra (-\epsilon,0)\times N$, such that on
$M\setminus K$:
\[ g=\varphi^*g_e.
\] 
\end{definition}

\begin{proof}[Proof of Theorem \ref{Theorem1}] 
As in the conical case, the Euler characteristic of
$M_r$ is constant and equal to $\chi(M)$. So it is enough to prove the
convergence of the integrals of transgression forms in \eqref{eq-4} for the
slices $\partialM _r\simeq\{r\}\times N $.

We will use the following terminology for double forms of type $(2,2)$ on $N$.
A form is called (purely) horizontal if its second component belongs to
$\Gamma(\pi^*\Lambda^2T^*B)$. It is called (purely) vertical its second
component belongs to $\Gamma(\Lambda^2V^*N)$.  It is a mixed form if its second
component belongs to $\Gamma(\pi^*T^*B\otimes V^*N\oplus V^*N\otimes
\pi^*T^*B)$. Clearly every $(2,2)$ form can be written as a sum of a purely
horizontal, a purely vertical and a mixed form.

The technical part of the proof is to decompose the curvature form of the
slice $\{r\}\times N$ for the metric $g_r:=r^2g^{V}(r)\oplus \pi^*g^B$ into
its horizontal, vertical and mixed components.  This is the object of
Proposition
\ref{Le2} below, according to which the curvature $F(\nabla^{g_r})$ for the
slice $\{r\}\times N$ with metric $r^2g^{V}(r)\oplus \pi^*g^B$ decomposes as
follows:
\[ F(\nabla^{g_r})=(A_0+A_2r^2+A_4r^4)+r^2(C_2+r^2C_4)+r^2
(D_2+r^2D_4)=X(r)+r^2Y(r)
\]
where $A_0,A_2,A_4,C_2,C_4,D_2,D_4$ are geometric quantities which depend
smoothly on $r$ down to $r=0$, and are constant when $g^{V}$ is constant
in $r$. Moreover, for all $i$, $A_i$ is purely horizontal, $D_i$ is purely
vertical and $C_i$ is mixed. We have $A_0=\pi^*F(\nabla^B)$ and 
$D_2=F(\nabla^{\VN  }_r)$, and this is all we need for subsequent computations.
Then
\begin{align*}X(r):=A_0+A_2r^2+A_4r^4,&&Y(r):=C_2+D_2+r^2(C_4+D_4)
\end{align*}
is a convenient separation of the terms.

Applying Lemma \ref{L1} yet again, we conclude that
\[\II^r=-\left(rg^{V}(r)+\frac{r^2}{2}\dot{g}^{V}(r)\right)=:-rZ
\]
where $Z$ is  a vertical $(1,1)$ double form. Let $b:=\dim B$
and $f:=2k-1-b$ be the dimension of the fiber of $\pi$. 
For the Berezin integrals, one has (taking into account that $r$ is negative):
\[\cB_{g_r}(\cdot)=\frac{1}{(-r)^f}\cB_{g^N}(\cdot).
\]
Then
\[ F(\nabla^{g_r})^j\wedge (\II^r)^{2k-1-2j}=\sum_{i=0}^j{\binom{j}i}\cdot
(-r)^{2k-1-2i} X^i \wedge (Y^{j-i}Z^{2k-1-2j}).
\]
 Hence
\begin{equation}\label{eqed2} \cB_{g_r}\left(F(\nabla^{g_r})^j\wedge
\II^{2k-1-2j}\right)=\sum_{i} {\binom{j}i} (-r)^{b-2i}\cB_{g^N}\left(X^i \wedge
(Y^{j-i}Z^{2k-1-2j})\right).
\end{equation}
Notice that $X^i$ is a purely
horizontal double form of bi-degree $(2i,2i)$, hence it vanishes if $2i>b$.
On the other hand, for $2i<b$ all forms $\omega=\cB_{g^N}(X^i \wedge
(Y^{j-i}Z^{2k-1-2j}))$ have a finite limit when $r\ra 0$. Therefore only the
term
$2i=b$ survives in the sum \eqref{eqed2} when $r\ra 0$. 

In conclusion, if $b$ is odd, the limit is $0$. If $b$ is even, we get in the
limit
\[{\binom{j}{b/2}}
\cB_{g^N}\left(X(0)^{b/2}Y(0)^{j-b/2}Z(0)^{2k-1-2j}\right).
\] 
Now $Y(0)=C_2+D_2$ and $C_2$ is a mixed term. Since $X(0)^{b/2}$ is a purely
horizontal form of maximal bi-degree, wedging with it will kill all terms from
$Y(0)^{j-b/2}$ containing a
horizontal component. Hence only $D_2^{j-b/2}$ will survive.
We are left with
\[ {\binom{j}{b/2}}\cB_{g^N}\left((\pi^*F(\nabla^B))^{b/2}
F(\nabla^{\VN  })^{j-b/2}(g^V)^{2k-1-2j}\right).
\]
Multiplying with $c(j,k)$,  integrating and summing over $0\leq i:=j-b/2 \leq
(f-1)/2$ gives the result, since $k-j-1=(f-1)/2-i$.
\end{proof}

\begin{example} \label{examp1} Let $\pi: E\ra B$ be a Euclidean vector bundle 
of  rank $2k$ endowed with a metric connection $\nabla$. Then  $\pi^*\nabla$ 
and the tautological section $\tau$ determine on $SE:=\{v\in E~|~|v|=1\}$ a
transgression form $\TPf(\pi^*\nabla,\tau)$ of degree $2k-1$ with the
property:
\begin{equation}\label{trsp}
\frac{1}{(2\pi)^k}\int_{SE/B}\TPf(\pi^*\nabla,\tau)\equiv 1
\end{equation}
when the fibers of $SE\ra B$ are oriented via the interior normals. This reduces
immediately to Example \ref{Exemp1}  (see also Remark \ref{nrem}).

Suppose now that $B$ is a {compact} submanifold of a closed Riemannian manifold
$\hat{M}$, both of even dimension. The normal
bundle $\nuB $ inherits a metric which is obviously a model edge metric with
$N=S(\nuB )$. Assume for the moment that the normal exponential map induces  an
isometry $D_{\epsilon}(\nuB )\ra U$ onto a neighborhood $U$ of $B$ where
$D_{\epsilon}(\cdot)$ is the disk bundle of radius $\epsilon$. Let
$M^{\circ}:=\hat{M}\setminus B$. Then  using \eqref{trsp}, Theorem
\ref{Theorem1} together with the classical Gauss-Bonnet turns into
\begin{equation}\label{topeq} \chi(M^{\circ})=\chi(\hat{M})-\chi(B).
\end{equation}
Clearly this relation is also a topological consequence of Mayer-Vietoris for
the cover $\{M^{\circ},U\}$ of $\hat{M}$. The same identity holds when $\dim{B}$
is odd, albeit in that case $\chi(B)=0$.

It turns out that \ref{Theorem1} continues to hold {\it{ad litteram}} if $B$ is
totally geodesic recovering once again \eqref{topeq}. In the general case, we
will turn the tables around. We will see in Theorem \eqref{tvsv} that the metric
on $M^{\circ}$, which, by the way, can be seen as the interior of the oriented
blow-up of $B$ is, in a neighborhood of the boundary,  a first order
perturbation of the model metric on $(-\epsilon,0]\times S(\nu B)$. Once we will
know that a Gauss-Bonnet formula \eqref{eqintro3} holds for such perturbations,
the topological statement \eqref{topeq} will serve to conclude that the integral
over $B$ equals $\chi(B)$. 

\end{example}

\begin{example} \label{examp2} Here is a more general situation when the
integral of the
transgression form is independent of the fiber. Let $P\ra B$ be
a principal bundle with structure group $G$. Suppose $G$ acts by isometries on a
Riemannian manifold $F$. Let $N:=P\times_G F$ be the associated fiber bundle
over $B$
({This is another way of saying that the fiber bundle with fiber
$F$ has transition maps taking values in $G\subset \Isom(F)$}).
Then the
vertical bundle $\VN  $ inherits a Riemannian metric, since $\VN  \simeq
P\times_G TF$
with $G$ acting on $TF$ via the differentials of the isometries. Since $TF$ has
a metric to start with and $G$ preserves it, we obtain a metric on
$\VN $.

Any $G$-principal connection $\omega\in \Omega^1(P;\mathfrak{g})$ gives rise to
a parallel transport via isometries between the fibers of $N\ra B$. Clearly the
transgression form $\TPf(N_b,g^{N_b},1)$ of a fiber $N_b$ obtained from the
conical metric $dr^2\oplus r^2g^{N_b}$ on $(-1,1)\times N$ depends only on the
isometry class of the metric $g^{N_b}$. Therefore in the situation when all the
fibers are isometric, the integral will be constant.
\end{example}

\begin{rem} One might ask what happens when $\dim{M}=2k+1$ is odd with an edge
singularity. If we look at $M_r$ which is a compact manifold of odd dimension
with boundary then by Lefschetz Duality one gets that
$\chi(M_r)=\frac{1}{2}\chi(\partialM _r)$. 

Now, $\chi(\partialM _r)=\int_{\partialM _r}\Pf(\nabla^{g_r})$ is constant with
respect to $r$. If one uses as above the decomposition of $F(\nabla^{g_r})$ into
its horizontal, mixed and vertical components then for  ${B}$ even dimensional 
one gets
\[\chi(N)= \lim_{r\ra 0}\frac{1}{(2\pi)^k(2k)!}\int_{\partialM _r}
F(\nabla^{g_r})^{2k}=\int_{B}\Pf(g^B)\int_{N/B}\Pf(g^V)=\chi(B)\chi(F)\]
while for odd  $\dim{B}$ one gets zero. We recover thus a Riemannian-geometric
proof of the multiplicativity of Euler characteristic in fibrations.
\end{rem}

\subsection{The curvature form of a Riemannian submersion}

In order to completely describe the decomposition of the curvature form
$F(\nabla^{g_r})$ into its vertical, horizontal and mixed components, 
we set $u:=r^2$, and consider the adiabatic deformation 
of the metric on $N$:
\[ h_u:=g_u^{V}\oplus u^{-1}\pi^*g^B.
\]
In this section we are interested in $uh_u$ but then in terms of curvature
\emph{forms}  one has:
\[ F{(\nabla^{uh_u}})=uF(\nabla^{h_u})
\]  
since the Levi-Civita connection of $uh_u$ and $h_u$ are the same. The reason
for working with $h_u$ is that we can make use of the results of \cite{BGV}, Ch.
10.
 
To begin with, let us notice that the family of vertical connections
$\nabla^{\VN }(u)$
resulting from the projections of the Levi-Civita connections $\nabla^{h_u}$ has
a limit $\nabla^{\VN }(0):=\displaystyle\lim_{u\ra 0}\nabla^{\VN }(u)$
and this limit is the projection of the Levi-Civita connection of $g^N$ (see
(\ref{gN})) onto
$\VN $.
This follows from the Koszul formula (see also Prop. 10.2 in \cite{BGV}).

Define, using the Ehresmann connection, the following family of connections on
$TN\ra N$:
\[ \nabla^{\oplus}_u:=\nabla^{\VN }(u)\oplus
\pi^*\nabla^B\longrightarrow \nabla^{\oplus}:=\nabla^{\VN }\oplus
\pi^*\nabla^B.\] 
\begin{rem} One should not confuse $\nabla^{\HN }$, the result of
projecting $\nabla^{h_u}$ onto $\HN $, with $\pi^*\nabla^B$.
\end{rem}

For $u\neq0$, let
$\tau_u:\Lambda^2T^*N\ra\End^-(TN)$ be the bundle morphism:
\[
\tau_u(\omega_1\wedge\omega_2)(\xi)=\omega_2(\xi)\omega_1^{\sharp_u}
-\omega_1(\xi)\omega_2^{\sharp_u}.
\]
The notation $\sharp_u$ represents the $h_u$-metric dual. Notice that $\tau_u$
is the inverse of 
\begin{align*}(\tau_u)^{-1}:\End^-(TN)\ra \Lambda^2T^*N,&&
(\tau_u)^{-1}(A)(\xi_1,\xi_2)=h_u(\xi_1,A\xi_2).
\end{align*}
We can write (see Prop. 10.6 in \cite{BGV}):
\[\nabla^{h_u}-\nabla^{\oplus}_u=\tau_u(\omega_u)
\]
for $u\neq 0$, where $\omega_u:TN\ra \Lambda^2T^*N$ is defined by
\begin{align*}
\omega_u(X)(Y,Z)=\hat{S}_u(X,Y,Z)-\hat{S}_u(X,Z,Y)-\hat{\Omega}_u(X,Z,Y)+\hat{
\Omega}_u(X,Y,Z)-\hat{\Omega}_u(Y,Z,X).
\end{align*}
We recall the definitions of  $\hat{S}_u$ and $\hat\Omega_u$ (both differ by a
sign compared with Section 10.1 in \cite{BGV}):
\begin{align*}
\hat{\Omega}_u\in \Gamma(\HN ^*\otimes \HN ^*\otimes \VN ^*),&&
\hat{\Omega}_u(X,Y,Z)={}&\frac{1}{2} g^V_u([X,Y]^{v},Z),\\
\hat{S}_u\in\Gamma(\VN ^*\otimes \VN ^*\otimes \HN ^*),&&
\hat{S}_u(X,Y,Z)={}&g^V_u(Y,[Z,X]^v-(\nabla^{\VN }(u))_ZX).
\end{align*}
where superscript $^v$ indicates projection onto the vertical component. Notice
that both $\hat{\Omega}_u$ and $\hat{S}_u$ have well-defined limits
when $u\ra 0$. We conclude that $\omega_u$ has a well-defined limit $\omega_0$
when $u\ra 0$.

We look at the curvature tensors now. We get:
\begin{equation}\label{curveq1}
F(\nabla^{h_u})=F(\nabla^{\oplus}_u)+[\nabla^{\oplus}_u,\tau_u(\omega_u)]
+\tau_u(\omega_u)\wedge\tau_u(\omega_u).
\end{equation}

Notice that for a fixed $u$, $\nabla^{\oplus}_u$ is $h_u$-metric compatible
since 
$\nabla^{\VN }(u)$ preserves $g^V(u)$ and
$\pi^*\nabla^B$ preserves $\pi^*g^B$. As a consequence, the morphism 
$\tau_u:\Lambda^2T^*N\ra\End^-(TN)$ is
parallel with respect to the connection $\nabla^{\oplus}_u$ for every $u$.
Therefore
\begin{equation}\label{curveq2}[\nabla^{\oplus}_u,\tau_u(\omega_u)]
=\tau_u(\nabla^{\oplus}_u\omega_u).
\end{equation}
where on the right $\nabla^{\oplus}_u$ is the extension on tensors of
$\nabla^{\oplus}_u$. It preserves the type of a double form, i.e., it takes
purely horizontal to purely horizontal, etc.

Due to the fact that $\nabla^{\HN }\neq \pi^*\nabla^{B}$, $\omega_u$ is not a
mixed form, which means that $\tau_u(\omega_u)$   has a certain diagonal
component. In fact we can write:
\begin{equation}\label{eq91} \omega_u=\tilde{\omega}_u+\omega^h_u
\end{equation}
where $\tilde{\omega}_u$ is made exclusively of mixed terms while $\omega^h_u$
is a purely horizontal term with:
\[ \tilde{\omega}_u:=(\tau_u)^{-1}(\nabla^{h_u}-\nabla^{\VN }(u)\oplus
\nabla^{\HN }(u))
\]
and
\[\omega^h_u:=(\tau_u)^{-1}(\nabla^{\VN }(u)\oplus
\nabla^{\HN }(u)-\nabla^{\oplus}_u).
\]
We used $\nabla^{\HN }(u)$ for the horizontal orthogonal projection of
$\nabla^{h_u}$ which does not coincide with $\pi^*\nabla^B$. Instead, we have
the following.

\begin{lem} \label{nl1} Let $\pi:P\ra B$ be a Riemannian submersion, and
$\nabla^{\HP }$ the orthogonal projection of the Levi-Civita connection onto
$\HP \simeq \pi^*TB$. Let 
\begin{align*}
\Omega:\HP \times \HP \ra \VP ,&& \Omega(X,Y)=P^{\VP }[X,Y]
\end{align*}
be the curvature of the Ehresmann connection (a
bundle morphism), and $\widetilde{\Omega}:\VP \times \HP \ra \HP $ the unique
bundle
morphism that satisfies
\begin{align*}  
\langle\widetilde{\Omega}(X,Y),Z\rangle=\langle X,\Omega(Y,Z)\rangle,&&
(\forall) Z\in \Gamma(\HP ).
\end{align*}
Then, for all $X\in \Gamma(TP), \;Y\in \Gamma(\HP )$,
\[
\nabla^{\HP }_XY-(\pi^*\nabla^B)_XY=\frac{1}{2}\Omega(P^{\HP }(X),Y)-\frac{1}{2}
\widetilde{\Omega}(P^{\VP }(X),Y).
\] 
In particular
\begin{align*}
\langle\nabla^{\HP }_XY,Z\rangle-\langle\pi^*\nabla^B_XY,Z\rangle
=&-\frac{1}{2}\langle P^{\VP }(X),[Y,Z]\rangle,& (\forall) Y,Z\in \Gamma(\HP ).
\end{align*}
\end{lem}
\begin{proof} It is well-known (see \cite{Pe}, pp.\ 82) that if $X$ and $Y$ are
horizontal lifts of vector fields $\overline{X},\overline{Y}$ on $B$ then
\[ \nabla_X^PY=\pi\circ
{\nabla_{\overline{X}}^B\overline{Y}}+\frac{1}{2}\Omega(X,Y).
\]
In other words, for this kind of vector fields we have:
\begin{equation}\label{conrel1} \nabla_X^PY=(\pi^*\nabla^B)_XY+\frac{1}{2}
\Omega(X,Y).
\end{equation}
It is easy to extend  equation \eqref{conrel1} to vector fields $X=fX_1$ and
$Y=gY_1$ where $X_1$ and $Y_1$ are horizontal lifts and $f,g\in C^{\infty}(P)$.
This means that \eqref{conrel1} holds for all $X,Y\in \Gamma(\HP )$.
 
On the other hand, for $X\in \Gamma(\VP )$  and $Y,Z$ horizontal lifts, one has 
\[ 2\langle \nabla_X^PY,Z\rangle=\langle [X,Y],Z\rangle-\langle
[Y,Z],X\rangle+\langle [Z,X],Y\rangle+X\langle Y,Z\rangle=-\langle
[Y,Z],X\rangle,
\]
the reason being that $[X,Y]=0=[Z,X]$ (see Lemma 10.7 in \cite{BGV}).
Since in this case $\pi^*\nabla^B_XY=0$  we get
\begin{equation}\label{conrel2} \langle \nabla_X^PY,Z\rangle
-\langle\pi^*\nabla_XY,Z\rangle=-\frac{1}{2}\langle [Y,Z],X\rangle
\end{equation}
and the relation holds also for $Y=gY_1$ and $Z=hZ_1$ with $Y_1$ and $Z_1$
horizontal lifts and $g,h\in C^{\infty}(P)$. This means that \eqref{conrel2}
holds for all $X\in \Gamma(\VP )$, $Y,Z\in \Gamma(\HP )$.
\end{proof}

According to Lemma \ref{nl1}, for $X\in \Gamma(TN)$ and $Y,Z\in
\Gamma(\HN )=\Gamma(\pi^*TB)$ we have:
\[\left\langle\left(\nabla^{\VN }(u)\oplus
\nabla^{\HN }(u)-\nabla^{\oplus}_u\right)_XY,Z\right\rangle=\omega^h_u(X)(Y,
Z)=-1/2 h^{V}_u(P^{\VN }(X),[Y,Z])
\]
and thus $\omega^h_u$ has a finite limit when $u\ra 0$. Since $\omega_u$
has a limit we deduce from \eqref{eq91} that $\tilde{\omega}_u$ has a limit
when $u\ra 0$. We conclude that
\[\nabla^{\oplus}_u\omega_u=\nabla^{\oplus}_u\tilde{\omega}_u+\nabla^{\oplus}
_u\omega^h_u
\]
is a decomposition into a purely mixed term and a purely horizontal term since
$\nabla^{\oplus}_u$ preserves the type of the form. Both sides have a
well-defined limit when $u\ra 0$.

\vspace{0.2cm}
In order to control
$(\tau_u)^{-1}(\tau_u(\omega_u)\wedge\tau_u(\omega_u))$ we need to take a closer
look
at $\tau_u$. Since for every $\eta\in \Omega^1(TN)$ we have
\[ \eta^{\sharp_u}=(\eta^v)^{\sharp^v_u}+u(\eta^h)^{\sharp^h}
\] 
where the decomposition $\eta=\eta^v+\eta^h$ is independent of $u$ and
$\sharp^v_u$ is the
$g^V_u$-metric dual while $\sharp^h$ is the $\pi^*g^B$-metric dual, we get:
\[\tau_u=\tau_0^u+u\tau_0',
\]
where
\begin{align*}
\tau_0^u:\Lambda^2T^*N\ra \Hom(TN,\VN ),&&
\tau_0^u(\omega_1\wedge\omega_2)(\xi)=\omega_2(\xi)(\omega_1^v)^{\sharp^v_u}
-\omega_1(\xi)(\omega_2^v)^{\sharp^v_u}\\
\tau_0':\Lambda^2T^*N\ra \Hom(TN,\HN ),&&
\tau_0'(\omega_1\wedge\omega_2)(\xi)=\omega_2(\xi)(\omega_1^h)^{\sharp^h}
-\omega_1(\xi)(\omega_2^h)^{\sharp^h}.
\end{align*}
\begin{rem} \label{remtau} 
Notice that 
\begin{itemize}
\item If $\omega_1$ or $\omega_2$ is
horizontal, then $\tau_0^u(\omega_1\wedge\omega_2)\xi=0$ for $\xi$ vertical. 
\item If $\omega_1$ or $\omega_2$ is vertical then
$\tau_0'(\omega_1\wedge\omega_2)\xi=0$ for $\xi$ horizontal.
\item If $\xi$ is vertical but $\omega_1$ and $\omega_2$ are both horizontal
then
$\tau_0'(\omega_1\wedge\omega_2)\xi=0$. 

\end{itemize}
\end{rem}
Clearly $\tau_0(u)$ has a
finite limit as $u\ra 0$. 
Define $\gamma_u:\Lambda^2TN\ra \Lambda^2 T^*N$:
\[ \gamma_u:=(\tau_u)^{-1}(\tau_u(\omega_u)\wedge\tau_u(\omega_u)).
\]
More explicitly,
\begin{align*}
\gamma_u(a_1,a_2)(\xi_1,\xi_2)={}&h_u\big(\xi_1,
\tau_u(\omega_u(a_1))\tau_u(\omega_u(a_2))-\tau_u(\omega_u(a_2))\tau_u(\omega_u(
a_1))\xi_2\big)\\
={}&h_u\big(\tau_u(\omega_u(a_2))\xi_1,
\tau_u(\omega_u(a_1))\xi_2\big)-h_u\big(\tau_u(\omega_u(a_1))\xi_1,
\tau_u(\omega_u(a_2))\xi_2\big)\\
={}&g^V_u\left(\tau_0^u(\omega_u(a_2))\xi_1,
\tau_0^u(\omega_u(a_1))\xi_2\right)-g^V_u\left(\tau_0^u(\omega_u(a_1))\xi_1
,\tau_0^u(\omega_u(a_2))\xi_2\right)\\
{}&+u\left[\pi^*g^B\big(\tau_0'(\omega_u(a_2))\xi_1,
\tau_0'(\omega_u(a_1))\xi_2)\big)-\pi^*g^B\big(\tau_0'(\omega_u(a_1))\xi_1,
\tau_0'(\omega_u(a_2))\xi_2)\big) \right].
\end{align*}
The last equality follows from the fact that $\tau_0(u)$ takes values in
$\VN $ and $\tau_0'$ takes values in $\HN $.

We define $(\omega\wedge\omega)_0^u:\Lambda^2TN\ra \Lambda^2 T^*N$,
\[
(\omega\wedge\omega)_0^u:=g^V_u\left(\tau_0^u(\omega_u(a_2))\xi_1,
\tau_0^u(\omega_u(a_1))\xi_2\right)-g^V_u\left(\tau_0^u(\omega_u(a_1))\xi_1
,\tau_0^u(\omega_u(a_2))\xi_2\right).
\]
By Remark \ref{remtau}, $\tau_0^u$ will take mixed forms
and purely horizontal forms into endomorphisms which vanish on vertical vectors.
Recall  \eqref{eq91}  by which $\omega_u$ is a sum of 
 mixed terms and purely horizontal terms. It follows that
$\tau_0^u(\omega_u(a_2))\xi$ is zero for $\xi$ vertical. We conclude that
$(\omega\wedge\omega)_0^u$ is a purely horizontal form.

Define also $(\omega\wedge\omega)_0'(u):\Lambda^2TN\ra \Lambda^2
T^*N$ by
\[(\omega\wedge\omega)_0'(u):=\pi^*g^B\big(\tau_0'(\omega_u(a_2))\xi_1,
\tau_0'(\omega_u(a_1))\xi_2)\big)-\pi^*g^B\big(\tau_0'(\omega_u(a_1))\xi_1,
\tau_0'(\omega_u(a_2))\xi_2)\big).
\]
Then
\begin{equation}\label{curveq3}
\gamma_u=(\omega\wedge\omega)_0^u+u(\omega\wedge\omega)_0'(u).
\end{equation}
We will use the same notation $F(\nabla^{h_u})$ for the curvature forms
$(\tau_u)^{-1}(F(\nabla^{h_u}))$
and $F(\nabla^{\oplus}_u)$ for $(\tau_u)^{-1}(F(\nabla^{\oplus}_u))$.
From \eqref{curveq1}, \eqref{curveq2} and \eqref{curveq3} we get the
following equality of $(2,2)$ double forms:
\[F(\nabla^{g_u})=F(\nabla^{\oplus}_u)+\nabla^{\oplus}
_u\omega_u+(\omega\wedge\omega)_0(u)+u(\omega\wedge\omega)_0'(u).
\]
 
The matrix decomposition $F(\nabla^{\oplus}_u)=F(\nabla^{\VN }(u))\oplus
F(\pi^*\nabla^B)$ translates into the equality of $(2,2)$ double forms for the
metric $h_u$:
\[ F(\nabla^{\oplus}_u)=F(\nabla^{\VN }(u))+u^{-1}\pi^*F(\nabla^B).
\]

We finally look at the decomposition for $(\omega\wedge\omega)_0'(u)$.  Use
\eqref{eq91} to get
\[(\omega\wedge\omega)_0'(u)=A_u^1+A_u^2+A_u^3+A_u^4,
\]
where
\[A_u^1(a_1,a_2)(\xi_1,\xi_2)=\pi^*g^B(\tau_0'(\tilde{\omega}_u(a_2))\xi_1,
\tau_0'(\tilde{\omega}_u(a_1))\xi_2)-\pi^*g^B(\tau_0'(\tilde{\omega}
_u(a_1))\xi_1,\tau_0'(\tilde{\omega}_u(a_2))\xi_2)
\]
\[
A_u^4(a_1,a_2)(\xi_1,\xi_2)=\pi^*g^B(\tau_0'({\omega}_u^h(a_2))\xi_1,\tau_0'({
\omega}_u^h(a_1))\xi_2)-\pi^*g^B(\tau_0'({\omega}_u^h(a_1))\xi_1,\tau_0'({\omega
}_u^h(a_2))\xi_2)
\]
\[A_u^3(a_1,a_2)(\xi_1,\xi_2)=\pi^*g^B(\tau_0'({\omega}_u^h(a_2))\xi_1,
\tau_0'(\tilde{\omega}_u(a_1))\xi_2)-\pi^*g^B(\tau_0'({\omega}_u^h(a_1))\xi_1,
\tau_0'(\tilde{\omega}_u(a_2))\xi_2)
\]
\[A_u^2(a_1,a_2)(\xi_1,\xi_2)=\pi^*g^B(\tau_0'(\tilde{\omega}_u(a_2))\xi_1,
\tau_0'({\omega}_u^h(a_1))\xi_2)-\pi^*g^B(\tau_0'(\tilde{\omega}_u(a_1))\xi_1,
\tau_0'({\omega}_u^h(a_2))\xi_2).
\]
Now $A_u^1$ is purely vertical, $A_u^4$ is purely horizontal, and moreover one
can
check that $A_u^2$ and $A_u^3$ are  mixed. 
We have thus proved the following
\begin{prop}\label{Le2} The following equality of $(2,2)$ double forms holds 
\begin{align*}
F(\nabla^{h_u})={}&\left[u^{-1}\pi^*F(\nabla^B)+\nabla^{\oplus}
_u\omega^h_u+(\omega\wedge\omega)_0^u+uA_u^4\right]\\
{}&+\left[\nabla^{\oplus}_u\tilde{\omega}_u+uA_u^2+uA_u^3\right]+\left[F(\nabla^
{\VN }(u))+uA_u^1\right]
\end{align*}
where the sums in square brackets represent the purely horizontal, mixed, or 
purely vertical components.
All terms dependent on $u$ have a finite limit when $u\ra 0$.
\end{prop}
From  $F(\nabla^{uh_u})=uF(\nabla^{h_u})$ one gets the corresponding
decomposition for $F(\nabla^{uh_u})$.

\subsection{Horizontal variations of the model metric}

We close this section by discussing what happens when the model metric has the
following structure:
\begin{equation}\label{modeledge}
g_e=dr^2\oplus r^2g^V(r)\oplus \pi^*g^B(r)
\end{equation}
with $g^B(r)$ a smooth family of metrics on $(-\epsilon,0]$. Various types of
perturbations will be considered in Section \ref{Sec8}.
By reasoning exactly as in the proof of Theorem \ref{Theorem1} one can compute
the limits of transgression
forms. In order to state the result, we need to introduce more notation.

Let $(g_r)_{r\in (-\epsilon,\epsilon)}$ be a smooth family of metrics on a
smooth manifold $B$ of dimension $b$. Let $g:=g_0$ and $\dot{g}:=\frac{\partial
g}{\partial r}(0)$ and denote:
\[ Q_{i,b}(g_r):=\frac{1}{i! (b-2i)!}\cB_g\left(R^i\wedge \dot{g}^{b-2i}\right).
\]
\begin{theorem}\label{olst}
On a manifold with incomplete edge singularities of type \eqref{modeledge},
\begin{align*} 
(2\pi)^k \chi(M)={}&\int_{M}\Pf^g-
\sum_{(i,j)\in A_{k,b}}(-1)^{2k-b}\tilde{c}(k-j-1)\int_B\left(
Q_{i,b}(g^B(r))\int_{N/B}P_{j,f}\left(g^V\right)\right)\\
\intertext{where}
A_{k,b}:={}&\{(i,j)~|~0\leq i\leq j\leq k-1,\; i\leq b/2 \}.
\end{align*}
\end{theorem}
\begin{proof} One writes $\II=-(rZ+T)$ where $T=\dot{g}^B$ and notices first
that $ZT=TZ$. Then one ends up with a sum for fixed $0\leq j\leq k-1$
\[ \sum_{i}\sum_{l}(-1)^{f+1}r^{b-(2i+l)}{\binom{j}{i}}
{\binom{2k-2j-1}{l}}\cB_{g^N}\left(X^i(0)T^l(0)Y^{j-i}(0)Z^{2k-2j-1-l}(0)\right)
\]
where $X(0)$ and $T(0)$ are purely horizontal. Only when $2i+l=b$ one gets
something non-trivial. Multiply by $c(j,k)$ and sum to get the desired formula.
\end{proof}
\begin{cor} If $\dot{g}^B(0)\equiv 0$ one recovers the formula of Theorem
\ref{Theorem1}.
\end{cor}
Anticipating Section \ref{Sec8}, Theorem \ref{olst} is an example of a
Gauss-Bonnet formula for first order
perturbations of the model metric
\[ dr^2\oplus r^2g^V(r)\oplus \pi^*g^B(0)
\]
in the sense of Definition \ref{def10}.

 \section{Manifolds with fibered boundary}

The computations of the previous section allow us to address the Gauss-Bonnet
problem for another class of metrics. Assume again that $N$ fibers over $B$,
that we fix an Ehresmann
connection, a family of vertical metrics on the fibers
and a metric $g^B$ on $B$. The model fibered boundary metric on
$(1,\infty)\times N$ defined by this data is
\[ g_e^{\infty}:=dr^2\oplus g^V\oplus r^2 \pi^*g^B.
\]
We consider Riemannian manifolds $(M,g)$ (called \emph{manifolds with
fibered boundary} for which there exists a diffeomorphism $\varphi:M\setminus
K\ra
(1,\infty)\times N$ outside a compact set $K$ such that
\[ g=\varphi^*g_e^{\infty}.
\]
\begin{prop} Manifolds with fibered boundary are complete.
\end{prop}
\begin{proof} Outside a relatively compact set, $M$ is isometric to
$[r,\infty)\times N$ endowed with the metric 
$g_e^{\infty}$ for some $r\in\bR$. The projection onto $[r,\infty)$ is proper
because $N$ is compact. Moreover, this projection clearly decreases lengths of
vectors, hence of curves, hence it decreases distances (it is Lipschitz of
constant $1$). This is enough to imply that $[r,\infty)\times N$ is a complete
metric space, hence $M$ is also complete.
\end{proof}

\begin{proof}[Proof of Theorem \ref{th1.3}] 
The computations are similar to Theorem \ref{Theorem1} and based also
on Proposition \ref{Le2} where we set $u^{-1}=r^2$.   Let $g_r:=g^V\oplus
r^2\pi^*g^B=h_{u}$  be the metric of the slice. Write the decomposition in
purely horizontal, mixed and purely vertical terms as:
\[
F(\nabla^{g_r})=(r^2A_2+A_0+r^{-2}A_{-2})+(C_0+r^{-2}C_{-2})+(D_0+r^{-2}D_{-2})
\]
where  $A_2=\pi^*F(\nabla^B)$, $D_0=F(\nabla^{\VN })$. Then
\[ (\II^r)^{2k-1-2j}=-r^{2k-1-2j}(\pi^*g^{B})^{2k-1-2j}
\]
and $\cB_{g_r}(\cdot)={r^{-b}}\cB_{g^N}(\cdot)$ where  $g^N=g^V\oplus
\pi^*g^B$. Hence
\[ \cB_{g_r}\left(F(\nabla^{g_r})^{j}\wedge
(\II^r)^{2k-1-2j}\right)=-r^{f-2j}\cB_{g^N}\left(F(\nabla^{g_r})^j\wedge(\pi^*g^
{B
})^{2k-1-2j}\right).
\]
We look at the term $(r^2A_2)^l$ for some $l\leq j$ in the expansion of
$F^{\nabla^{g_r}}$. Now the horizontal component of the product
$F(\nabla^{g_r})^j\wedge(\pi^*g^{B})^{2k-1-2j}$ cannot have degree bigger than
$b$ in order to be non-zero. Hence
\[ 2l+2k-1-2j\leq b ~\Leftrightarrow~ 2l+f-2j\leq 0.
\]
 All the other terms in the expansion of $F(\nabla^{g_r})$ contribute with
non-positive powers of $r$. Hence in the expansion of
$r^{f-2j}\cB_{g^N}\left(F(\nabla^{g_r})^j\wedge(\pi^*g^{B})^{2k-1-2j}\right)$
one
ends up only with non-positive powers of $r$. 

If $b$ is even, the inequalities
are strict so all terms will vanish when $r\ra \infty$. 
If $b$ is odd, collecting the terms that correspond to $2l=2j-f$ (which
incidentally 
forces $j\geq f/2$) we get \eqref{gbt}.
\end{proof}

\begin{cor} Let $(M,g)$ be a manifold with fibered boundary. If the base $B$ of
the boundary 
fibration $N\ra B$ is an odd-dimensional sphere with the round metric, then
\[ \chi(M)-\chi(F)=\frac{1}{(2\pi)^k}\int_{M}\Pf^g. 
\]
\end{cor}
\begin{proof} The normal $\partial_r$ points
outside the sphere. The computations of Example \ref{Exemp1} apply
(see also Remark \ref{nrem}). This fits with the example when $M=\bR^n$ and $F$ 
reduced to a point.
\end{proof}

\section{Edge manifolds: perturbations of the model metrics} \label{Sec8}
There is one familiar situation not entirely covered by the models of Section
\ref{Emmm}, namely that of a submanifold $B$ in a Riemannian manifold $(M,g)$.
The spherical normal bundle $N:=S\nuB $ inherits a fiber bundle structure over
$B$ and an Ehresmann connection, induced by the Levi-Civita connection as
follows. Let $\pi:TN\ra N$ be the natural projection. The Levi-Civita connection
induces a connection on $\nuB $ and therefore one obtains a splitting $T\nu
B=\pi^*\nuB \oplus \pi^* TB$ into vertical and horizontal components where
$\pi:\nuB \ra B$ is the natural projection.  Now  $S(\nuB )\subset \nuB $  is a
hypersurface whose unit normal vector is vertical (i.e., it belongs to $\pi^*\nu
B$) relative to the previous decomposition. It follows that $TS(\nuB )$ splits 
into the direct sum of $\tau^{\perp}\subset \pi^*\nuB $ (the orthogonal
complement of the tautological section of $\pi^*\nuB \ra S(\nuB )$) and $\pi^*
TB$. 

On both $TB$ and the normal vector bundle $\nuB \ra B$ there are metrics induced
by $g$, hence $(-\epsilon,0)\times N$ inherits an edge singularity metric.
However, the original metric $g$ in a neighborhood  of $B$ is not necessarily
isometric to a model metric in the sense defined in  Section \ref{Emmm} since
the normal exponential map that gives rise to a tubular neighborhood for $B$ is
only an "infinitesimal" isometry at the $0$ section.

It is therefore natural to consider perturbations of the model edge metrics of
Section
\ref{Emmm}. 

We will consider a differentiable edge manifold, meaning a compact manifold $M$
with boundary $N$, such that $\pi:N\ra B$ is a locally trivial fibration.
Moreover
we assume the following data given:
\begin{itemize}
\item[(a)] a boundary defining function $r:M\ra (-\epsilon,0]$;
\item[(b)] an Ehresmann connection on $N$, i.e., a splitting $TN=\VN \oplus
\pi^*TB$
\end{itemize}

We can use $r$ in order to produce a collar neighborhood $U$ of  $N$
diffeomorphic with $(-\epsilon,0]\times N$ such that the obvious diagram
commutes:
\[ \xymatrix {U \ar[dr]^r \ar[rr]^{R}_{\widetilde{\qquad}}&& 
(-\epsilon,0]\times N  \ar[dl]^{p_1}\\
 &   (-\epsilon,0]   &}
\]
The differential of $R$ gives a diffeomorphism between $TM\bigr|_{U}$ and
$\bR\oplus \pi_2^*TN$, where $\pi_2:(-\epsilon,0]\times N\ra N$ is the
projection on the second factor. 

For our purposes, the edge manifold $M$ in the neighborhood $U$ will
be identified with $(-\epsilon,0]\times N$ while the tangent bundle to $M$ in a
neighborhood $U$ will be identified with $\bR\oplus \pi_2^*TN$.  The unit
generator of $\bR$ in this identification will be denoted $\partial_r$.

For the sake of brevity, we denote $U:=(-\epsilon,0]\times N$.

Consider the vector bundles $F:=\VN $ and $F':=\pi^*TB\oplus \bR$ over $N$. 
The Ehresmann connection induces a splitting
\[  \bR\oplus TN\simeq F\oplus F'.
\]
We use the projection $\pi_2:(-\epsilon,0]\times N\ra N$ to pull-back this
bundle to $U$ but rather than writing $\pi_2^*F$, $\pi_2^*F'$ we keep the
notation $F$, $F'$. We have thus in the neighborhood $U$ a splitting
\begin{equation} \label{sp1} TM\bigr|_{U}\simeq F\oplus F'
\end{equation}

The fundamental object of this section is the following bundle endomorphism
defined in terms of the splitting \eqref{sp1}:
\begin{align*} 
\varphi:TM\bigr|_{U}\ra TM\bigr|_{U},&& F\oplus
F'\ni(v,w)\stackrel{\varphi}{\longmapsto}(rv,w).
\end{align*}
Clearly, $\varphi$ is a bundle isomorphism only along $U^c:=U\setminus N$, i.e.,
for $r\neq 0$. 

The model edge degenerate metric is throughout this section:
 \[ h:=dr^2\oplus r^2g^{V}\oplus \pi^*g^B.\] 
The bilinear map
\begin{align*}
h^{\varphi}:TM\bigr|_{U^c}\times TM\bigr|_{U^c}\ra \bR,&&
h^{\varphi}(Y',Z'):=h(\varphi^{-1}(Y'),\varphi^{-1}(Z'))
\end{align*}
extends as a non-degenerate metric on $U$, and $\varphi$ becomes a bundle
isometry for $r\neq0$. Indeed, 
\[ h^{\varphi}=dr^2\oplus g^{V}\oplus \pi^*g^B.\]
\begin{theorem}\label{fundprop} 
The Levi-Civita connection $\nabla^h$ of the model metric has the
property that $\varphi \nabla^h \varphi^{-1}$ extends to a $h^{\varphi}$-metric
connection down to $r=0$.
\end{theorem}
\begin{proof} 
We will compare the Levi-Civita connection of $\nabla^h$ with the following
connection:
 \begin{equation}\label{nabla'} \nabla':=d\oplus
\left[\left(\frac{\partial}{\partial
r}+\frac{1}{r}\right)dr+\nabla^{\VN }\right]\oplus \pi_2^*\pi^*\nabla^B
\end{equation}
 on the vector bundle
 $TM\bigr|_{U^c}= \bR\oplus \pi_2^*\VN \oplus \pi_2^*\pi^*TB$ where
$\pi_2:(-\epsilon,0]\times N\ra N$ is the projection.

In \eqref{nabla'}, the connection $\nabla^{\VN }$ is the projection of the
Levi-Civita connection of any slice $\{r\}\times N$ onto $\pi_2^*\VN $. It does
not depend on $r$ since the projection of the
Levi-Civita connection of a Riemannian submersion onto the vertical bundle does
not depend on the choice of the horizontal metric (Prop. 10.2 in \cite{BGV}),
while the Levi-Civita connection of the slice $\{r\}\times N$ is the same for
the metrics $r^2g^{V}\oplus \pi^*g^B$ and $g^{V}\oplus
r^{-2}\pi^*g^B$.

We emphasize that the differential operator $\frac{\partial}{\partial
r}+\frac{1}{r}$ acts  on families of sections 
\[(Y_r)_{r\in(-\epsilon,0]}\in \Gamma(\VN )\] which can alternatively be seen as
sections of $\pi_2^*\VN $ (where $\pi_2:(-\epsilon,0]\times N\ra N$ is the
projection), while $\nabla^{\VN }$ is used to differentiate only in the $TN$
directions.

It follows from the Koszul relation (see \eqref{nry} and \eqref{nry1}) that the
$\pi_2^*\VN $ component of $\nabla'$ is actually the orthogonal projection of
$\nabla^h$ onto $\pi_2^*\VN $. This implies that $\nabla'$ is $h$-compatible
(as $\pi^*\nabla^B$ is clearly $\pi^*g^B$-compatible). As a consequence,
$\varphi\nabla'\varphi^{-1}$ is $h^{\varphi}$-compatible. 

It is easy to check that $\varphi\nabla'\varphi^{-1}$ extends smoothly to $r=0$
since
$\nabla^{\VN }$ commutes with multiplication by $r^{-1}$ and 
\[ \frac{\partial Y_r}{\partial r}+\frac{Y_r}{r}=\frac{1}{r}\frac{\partial
(rY_r)}{\partial r}.
\]
Moreover, $\varphi (d\oplus\pi^*\nabla^B)\varphi^{-1}=d\oplus\pi^*\nabla^B$,
since
$\varphi$ acts as the identity on $F'$.

In order for the $1$-form $\eta:=\nabla^h-\nabla'$ to have the property that
$\varphi\eta(X)\varphi^{-1}$ extends smoothly for every choice of $X\in
\Gamma(TM\bigr|_{U})$, it is enough that in the decomposition
\begin{equation}\label{md1} \eta(X):=\left(\begin{array}{cc} A_1(X)& A_2(X)\\
 A_3(X)& A_4(X)\end{array}\right): \begin{array}{c} F\\\oplus \\ F'
 \end{array}\ra  \begin{array}{c} F\\ \oplus \\ F'
 \end{array}
\end{equation}
the blocks $A_{1}(X)$, $A_4(X)$, $rA_2(X)$ and $r^{-1}A_3(X)$ extend smoothly 
all the way down to $r=0$.
 
Clearly $A_1\equiv 0$ since the orthogonal projections of $\nabla^h$ and
$\nabla'$ on $F$ coincide.
 
 Then metric compatibility implies for $Y\in \Gamma(F')$ and $Z\in \Gamma(F)$
\begin{align*}  r^2\langle A_2(X)(Y),Z\rangle_{\VN }={}&\langle
A_2(X)(Y),Z\rangle_h=-\langle Y,
A_3(X)(Z)\rangle_h\\
={}& -\langle Y, A_3(X)(Z)\rangle_{F'}=-\langle
A_3^T(X)(Y),Z\rangle_{\VN }
\end{align*}
 where the transpose $A_3^T$ is computed with respect to the metric
$h^{\varphi}$,
independent of $r$. Hence
  \[rA_2(X)=-r^{-1}A_3^T(X).\]
Thus, it is enough to prove that $r^{-1}A_3(X)$ extends smoothly to $r=0$.
 
To see that the remaining relations hold we look again at the Koszul relation: 
\[ 2\langle \nabla_X^hY,Z\rangle_h=\langle
[X,Y],Z\rangle_h-\langle[Y,Z],X\rangle_h+\langle [Z,X],Y\rangle_h+X\langle
Y,Z\rangle_h+Y\langle Z,X\rangle_h-Z\langle X,Y\rangle_h.
\]
For $X=\partial_r,$ $Y=Y_r\in\Gamma(\pi_2^*\VN )$, $Z=Z_r\in \Gamma(\pi_2^*\VN
)$,
\[ 2r^2\left\langle \nabla_{\partial_r}Y,Z\right\rangle_{g^{\VN
}}=r^2\left\langle
\frac{\partial{Y}}{\partial r},Z \right\rangle_{g^{\VN }}-r^2\left\langle
\frac{\partial Z}{\partial r},Y\right\rangle_{g^{\VN }}+\frac{\partial}{\partial
r}\left[r^2\langle Y,Z\rangle_{g^{\VN }}\right].
\]
We end up with
\[2r^2\left\langle \nabla_{\partial_r}Y,Z\right\rangle_{g^{\VN
}}=2r^2\left\langle
\frac{\partial Y}{\partial r},Z \right\rangle_{g^{\VN }}+2r\left\langle
Y,Z\right\rangle_{g^{\VN }}.
\]
Hence
\begin{equation}
\label{nry} \nabla_{\partial_r}Y=\frac{\partial Y}{\partial r}+\frac{Y}{r}
\end{equation}
Taking $X=X_r\in \Gamma(\pi_2^* TN)$  with $Y=Y_r\in \Gamma(\pi_2^*\VN )$,
$Z=Z_r\in \Gamma(\pi_2^*\VN )$ then clearly
\begin{equation}
\label{nry1} \langle\nabla_X^h Y,Z\rangle=\langle\nabla^{\VN }_XY,Z\rangle.
\end{equation}
One verifies easily that the orthogonal projection of $\nabla^h$ onto $\bR$, the
tangent bundle of the foliation via integral curves of $\partial_r$, is the
trivial connection
$d$.

Recall that $\pi^*\nabla^B$ is not the orthogonal projection of $\nabla^h$ onto
$\pi^*TB\simeq \HN $. Let $\nabla^{\HN }$ be this projection. It follows from
Lemma
\ref{nl1} for the Riemannian submersion $M\bigr|_{U}\ra B$ that for $X\in 
\Gamma(\bR\oplus\pi_2^*TN)$ and $Y,\; Z \in \Gamma(\pi_2^*\pi^*TB)$:
\[\langle\nabla^h_XY,Z\rangle_h=\langle\nabla_X^{\HN }Y,Z\rangle_h=\left\langle
\pi^*\nabla^B_XY,Z\right\rangle_h-\frac{1}{2}\left\langle P^{\bR\oplus
\VN }(X),[Y,Z]\right\rangle_h
\]

When $X=\partial_r$ since $[Y,Z]\in \Gamma(\pi_2^*TN)$ (one has a foliation via
hypersurfaces $\{r\}\times N$) the last term is zero.

When $X\in \Gamma(\pi_2^*TN)$ then 
\[\langle\nabla^h_XY,Z\rangle_h=\left\langle
\pi^*\nabla^B_XY,Z\right\rangle_h-\frac{r^2}{2}\left\langle
P^{\VN }(X),[Y,Z]\right\rangle_{g^{\VN }}
\]
and the right hand side is smooth at $r=0$. This describes the bottom block
diagonal component  of $A_4(X)$ in \eqref{md1} relative to the decomposition
$F'=\bR\oplus \pi^*TB$. The other diagonal block of $A_4$ is obviously $0$. The
off-diagonal terms of the skew-symmetric $A_4(X)$ are of type
\[ \langle\nabla_X \partial_r,Y\rangle_h \mbox{ and its negative }
\langle\nabla_XY,\partial_r\rangle_h
\] 
where $X\in \Gamma(\bR\oplus\pi_2^*TN)$, $Y\in\Gamma(\pi_2^*\HN )$. For
$X=\partial_r$ one gets obviously $0$ and  Lemma \ref{L1} gives for $X\in
\Gamma(\pi_2^*TN)$:
\[ \langle \nabla_X \partial_r,Y\rangle_h=\frac{1}{2}(L_{\partial_r}h)(X,Y)=r
\langle X,Y\rangle_{\VN }=0.\] 
In other words, if $\widetilde{\Omega}:\VN \times \HN \ra \HN $ is the morphism
induced by the curvature $\Omega$ of the Ehresmann connection of the Riemannian
submersion $\pi:N\ra B$ with the metric $g^{V}\oplus \pi^*g^B$ as in Lemma
\ref{nl1}, then  for $X\in \Gamma(\bR\oplus \pi_2^*TN)$, $Y\in \Gamma(F')$ one
has:
\[  A_4(X)(Y)=-\frac{r^2}{2}\widetilde{\Omega}(P^{\VN }(X),P^{\HN }(Y)).
\]
Finally,  for $Y\in \Gamma(\pi_2^*\VN )$, $Z\in \Gamma(F')$ and $X\in
\Gamma(\bR\oplus \pi_2^*TN)$ we compute 
 \[\langle A_3(X)(Y),Z\rangle_h=\langle \nabla_X^hY,Z\rangle_h.\]
For $X=\partial_r$, $Z\in \pi_2^*\HN $ one gets from the Koszul formula
 \begin{equation}\label{B31} 2\left\langle
\nabla_{\partial_r}Y,Z\right\rangle_h= \left\langle {\partial _rY},Z\right
\rangle_h-\left\langle{\partial_r Z},Y\right\rangle_h=0.
 \end{equation}
 The vanishing holds also for $X=\partial_r$, $Z=\partial_r$. 
 
 For $X\in \Gamma(\pi_2^*TN)$, $Z=\partial_r$ we get:
 \begin{equation}\label{B32}  \langle A_3(X)(Y),\partial_r\rangle=\langle
\nabla_X^hY,\partial_r\rangle_h={\II_{r}}(X,Y)=-r\langle X,Y\rangle_{\VN }.
 \end{equation}
 For $X\in \Gamma(\pi_2^*\VN )$, $Z\in \Gamma(\pi_2^*\HN )$ we get the relation:
 \begin{equation}\label{B33}\langle A_3(X)(Y),Z\rangle_{\HN }=-\langle Y,
P^{\VN }(\nabla_X^hZ)\rangle_{h}=r^2\langle
Y,P^{\VN }([Z,X])-\nabla_Z^{\VN }X\rangle_{\VN }.
\end{equation}
 For $X\in \Gamma(\pi_2^*\HN )$, $Z\in \Gamma(\pi_2^*\HN )$ we get the curvature
of
the Ehresmann connection:
 \begin{equation}\label{B34} \langle A_3(X)(Y),Z\rangle_{\HN }=-\langle Y,
P^{\VN }(\nabla_{X}^hZ)\rangle_h=-r^2\langle Y, \Omega(X,Z)\rangle_{\VN }.
 \end{equation}
It is now clear from \eqref{B31},  \eqref{B32}, \eqref{B33} and \eqref{B34}
that $\frac{A_3(X)}{r}$ extends for any smooth vector fields
$X,Y:(-\epsilon,0]\times N\ra TM\bigr|_{U}$.
\end{proof}

\begin{cor}\label{Cor72} The restriction to $TM\bigr|_{\partialM }=\bR\oplus
\pi_2^*\VN \oplus \pi_2^*\pi^* TB$ (i.e., to $r=0$) of the extended connection
$\varphi \nabla^h\varphi^{-1}$
coincides with the connection
 \begin{equation}\label{extat0} \left(\begin{array}{c}d\\
\frac{\partial}{\partial r} dr+\nabla^{\VN }\\
\pi^*\nabla^B\end{array}\right)+\left(\begin{array}{c |c |c} 0 &
\langle\bullet,\cdot \rangle_{\VN } &0\\
  -\langle \bullet,\cdot \rangle_{\VN } &0 &0\\
   0&0&0 \end{array}\right)
 \end{equation}
 where the matrix represents a $1$-form (the $\bullet$ entry) with values in
$\End(\bR\oplus \pi_2^*\VN \oplus \pi_2^*\pi^* TB)$.
 \end{cor}
 \begin{proof} The only non-trivial term in the difference $\varphi
(\nabla^h-\nabla')\varphi^{-1}$ comes from relation \eqref{B32}.
 \end{proof}
 \begin{cor}\label{corf} The Pfaffian $\Pf(\nabla^h)$ is a smooth form on $M$
down to the boundary $\{r=0\}$.
 \end{cor}
 \begin{proof} The map $\varphi:(TM\bigr|_{U^c},h)\ra
(TM\bigr|_{U^c},h^{\varphi})$ is a bundle isometry. Hence on $U^c$,
$\Pf(\nabla^h)$ is, up to a sign,   equal to
$\Pf(\varphi\nabla^h\varphi^{-1})$. 
 \end{proof}

We consider now a perturbation $g$ of $h$, i.e.,  a bilinear and symmetric form
on $TM$ that is degenerate only along $N$ in a sense  made precise in Definition
\ref{def10}.

 Clearly there exists an $h$-symmetric endomorphism $C\in
\Gamma(\End(TM\bigr|_{U^c}))$ such that
\begin{align*} 
g(X,Y)=h(CX,Y)=h(X,CY),&& (\forall) X,Y\in TM\bigr|_{U^c}.
\end{align*}

The next Lemma linking the two Levi-Civita connections is fundamental for our
computations.
\begin{lem}[Christoffel formula]\label{Crf} Let $\nabla^h$ and $\nabla^g$ be the
corresponding Levi-Civita connections on $TM\bigr|_{U^c}$. Then the $1$-form
$\omega:TM\bigr|_{U^c}\ra  \End(TM\bigr|_{U^c})$ defined by
\[ \omega(X)(Y)=\nabla^g_XY-\nabla_X^hY,
\]
satisfies:
\[
h(C\omega(X)(Y),Z)=\frac{1}{2}\left(h((\nabla_X^hC)Y,Z)+h((\nabla_Y^hC)X,
Z)-h((\nabla^h_Z C)X,Y\right).
\]
\end{lem}
\begin{proof} Notice first that due to the symmetry of the Levi-Civita
connections one has:
\begin{equation}\label{eq1s9} \omega(X)(Y)=\omega(Y)(X)
\end{equation}
and therefore $C\omega(X)(Y)=C\omega(Y)(X)$. Then from
\[ Xh(Y,CZ)=h(\nabla^h_XY,CZ)+h(Y,\nabla^h_X(CZ)) \;\mbox{and}
\]
\[ Xg(Y,Z)=g(\nabla^g_XY,Z)+g(Y,\nabla^g_XZ)
\]
which translates into
\[ Xh(Y,CZ)=h(\nabla^g_XY,CZ)+h(Y,C\nabla^g_XZ)
\]
one gets by subtraction:
\[ h(\nabla^h_XY-\nabla^g_XY,CZ)=h(Y,C\nabla^g_XZ-\nabla_X^h(CZ)).
\]
Taking $\nabla^h_X(CZ)=C(\nabla_X^hZ)+(\nabla_X^hC)(Z)$  we get:
\[ h(\omega(X)(Y),CZ)+h(Y,C\omega(X)(Z))=h(Y,(\nabla^h_X C)(Z)).
\]
or
\begin{equation}\label{eq2s9} \omega(X)^TC+C\omega(X)=\nabla^h_X C.
\end{equation}
Notice that the system  \eqref{eq1s9} and \eqref{eq2s9} has a unique solution
for $C\omega(X)$ due to the well-known fact that a trilinear map which is
symmetric in the first two variables and anti-symmetric in the last two
variables is zero. Finding this solution is simple linear algebra. 
\end{proof}

 We know already that $\varphi\nabla^h\varphi^{-1}$ extends to $r=0$. Let 
\[ C=\begin{bmatrix} C_1& C_2 \\
 C_3& C_4\end{bmatrix}: \begin{bmatrix} F\\ F'\end{bmatrix}\ra
\begin{bmatrix} F\\ F'\end{bmatrix}
\] 
be the block decomposition of $C$. Then 
\[C^{\varphi}:=\varphi C\varphi^{-1}=\begin{bmatrix} C_1& rC_2 \\
 r^{-1}C_3& C_4\end{bmatrix}
\] 
is symmetric with respect to the $h^{\varphi}$-metric. In other words $C_3=r^2
C_2^T$ where the transpose is computed with respect to $h^{\varphi}$.
We have the following obvious remark.
\begin{lem}\label{il1} If $\varphi C\varphi^{-1}$ extends smoothly to
$TM\bigr|_{U}$, then
$g^{\varphi}(\cdot,\cdot):=g(\varphi^{-1}(\cdot),\varphi^{-1}(\cdot))$ extends
and
\[ g^{\varphi}(\cdot,\cdot)=h^{\varphi}(C^{\varphi}\cdot,\cdot).
\]
\end{lem}
The morphism $\varphi C\varphi^{-1}$ controls the degenerations we are
interested in. 
By Lemma \ref{il1}, saying that
\[ C=I+ f(r)\varphi^{-1}D\varphi
\]
where $D$ is smooth at $r=0$ and $h^{\varphi}$-symmetric and $f$ smooth and
vanishing at $0$ is equivalent to 
\[ g^{\varphi}(\cdot,\cdot)=h^{\varphi}(\cdot,\cdot)+f(r)\alpha(\cdot,\cdot)
\]
for some $\alpha(\cdot,\cdot)$ smooth, bilinear and symmetric on $TM\bigr|_{U}$.

\begin{definition} \label{def10} A perturbation of first (respectively second)
order of $h$ is a bilinear, positive, symmetric $g:TM\bigr|_{U^c}\times
TM\bigr|_{U^c}$ such that the endomorphism $C$ above satisfies:
\[ C=I+r\varphi^{-1}D\varphi,\quad \mbox{resp.} \quad
C=I+r^2\varphi^{-1}D\varphi,
\]
where $D$ is a smooth endomorphism of $TM\bigr|_{U}$, symmetric in the
$h^{\varphi}$ metric. 

Equivalently for $p=1$ (resp. $p=2$)
\[ g^{\varphi}(\cdot,\cdot)=h^{\varphi}(\cdot,\cdot)+r^p\alpha(\cdot,\cdot)
\]
where $\alpha$ is bilinear, symmetric and smooth on $TM\bigr|_{U}$.
\end{definition}
 
\begin{lem}\label{il2} \[\varphi(\nabla^h C)\varphi^{-1}=(\varphi\nabla^h
\varphi^{-1})(\varphi C\varphi^{-1}).\] 
\end{lem}
\begin{proof} It follows from the next equalities that hold for any $X$ and $Y$:
\begin{align*}\varphi(\nabla^h
_XC)\varphi^{-1}(Y)={}&\varphi(\nabla^h_X(C\varphi^{-1}(Y))-\varphi
C(\nabla^h_X(\varphi^{-1}(Y)))
\\
(\varphi\nabla^h_X \varphi^{-1})(\varphi
C\varphi^{-1})(Y)={}&\varphi(\nabla^h(\varphi^{-1}\varphi C
\varphi^{-1}(Y)))-\varphi C\varphi^{-1}(\varphi \nabla^h
(\varphi^{-1}(Y))).\qedhere
\end{align*}
\end{proof}

\begin{theorem}\label{Te1}
Let $g$ be a perturbation of a model edge metric $h$.
\begin{itemize}
\item[(i)] For perturbations of first order, the connection $\varphi\nabla^g
\varphi^{-1}$ extends at $r=0$.
\item[(ii)] For perturbations of second order the connection the extension of
$\varphi\nabla^g \varphi^{-1}$ coincides on $TM\bigr|_{\partialM }$ with
$\varphi \nabla^h \varphi^{-1}$.
\end{itemize}
\end{theorem}
\begin{proof}  Let $C^{\varphi}:=\varphi C \varphi^{-1}$, $Y':=\varphi(Y)$,
$Z':=\varphi(Z)$, $\nabla^{\varphi}:=\varphi\nabla^h\varphi^{-1}$,
 $h^{\varphi}(\cdot,\cdot):=h(\varphi^{-1}(\cdot),\varphi^{-1}(\cdot))$, 
and $\omega(X)^{\varphi}:=\varphi \omega(X)\varphi^{-1}$.
The Christoffel formula (Lemma \ref{Crf}) can be written using Lemma
\ref{il2}:
\begin{equation}
\label{eq100}\begin{split}
\lefteqn{2h^{\varphi}\left(C^{\varphi}\omega(X)^{\varphi}(Y'),Z'\right)=
h^{\varphi}\left((\nabla^{\varphi}_X C^{\varphi})(Y'),Z'\right) +}&\\
&+ h^{\varphi}\left(
(\nabla_{\varphi^{-1}(Y')}^{\varphi}C^{\varphi})(\varphi(X))),Z'\right)  -
h^{\varphi}\left((\nabla_{\varphi^{-1}(Z')}^{\varphi}C^{\varphi})(\varphi(X))),
Y'\right).
\end{split}\end{equation} 
We deduce from this formula that, in order to show that $\varphi
\omega(X)\varphi^{-1}$ extends for perturbations of first order, it is enough to
show that
 \[\nabla_{\varphi^{-1}(Y')}^{\varphi}C^{\varphi}=\nabla^{\varphi}_{\varphi^{-1}
(Y')}(rD)\] extends for all choices of $Y'$, since the first term in the sum
(r.h.s.\ of \eqref{eq100})) extends anyway.
 
 The only situation when the extension is not \emph{a priori} clear is when $Y'\in
\Gamma(\pi_2^*\VN )$. Then $\varphi^{-1}(Y')=\frac{Y'}{r}$. But we can use now
that $Y'(r)=0$ and therefore
 \[\nabla^{\varphi}_{\varphi^{-1}(Y')}(rD)=\nabla^{\varphi}_{Y'}(D),
 \]
 and the later term extends.

Since $C^{\varphi}\ra 0$ when $r\ra 0$, in order to show that $\varphi
\omega(X)\varphi^{-1}$ extends by $0$ for perturbations of second order we need
to check that 
\[ \lim_{r\ra 0}\nabla^{\varphi}_X C^{\varphi}= \lim_{r\ra 0}\nabla^{\varphi}_X(
r^2D)=0
\]
\[\lim_{r\ra 0}\nabla_{\varphi^{-1}(Y')}^{\varphi}C^{\varphi} =\lim_{r\ra
0}\nabla_{\varphi^{-1}(Y')}^{\varphi}(r^2D)=0
\]
for all choices of $X$ and $Y'$. If either $X=Y'=\partial_r$ then since
$\varphi^{-1}(\partial_r)=\partial_r$ the two  limits are identical and clearly
equal to $0$. When $Y'\in \Gamma(\pi_2^*\VN )$ then the same idea as in the
first
order perturbations apply.
\end{proof}

\subsection{The Riemannian metric in a neighborhood of a submanifold} 

The purpose of this Section is to prove that the degenerate metric on the
oriented blow-up space of a  submanifold inside a Riemannian manifold is a
first-order perturbation of a canonical model edge degenerate metric.

Let $B\subset M$ be a compact submanifold in a Riemannian manifold $(M,g)$. Let 
$\nuB \subset TX_{|B}$ be the normal bundle, $\pi:S(\nuB )\to B$ the unit sphere
bundle 
inside $\nuB $, and 
\begin{align*}
\exp:S(\nuB )\times [0,\infty)\to M,&&(v_x,r)\mapsto \exp_x(rv_x)
\end{align*}
the geodesic exponential map in normal directions to $B$. 
This map defines a diffeomorphism from $S(\nuB )\times (0,\epsilon)$ 
to the complement of $B$ inside its $\epsilon$-neighborhood. 
The function $r$ becomes the distance function to $B$. In fact, replacing the
$\epsilon$-neighborhood of $B$
with $S(\nuB )\times [0,\epsilon)$ amounts precisely to constructing the (real)
blow-up of $M$ along $B$.

The normal bundle $\nuB $ inherits itself a metric which makes the canonical
projection $\pi: \nuB \ra B$ a Riemannian submersion. The Ehresmann connection
here is just the normal connection on $B$ induced from the Levi-Civita 
connection of $M$. One can use the blow-down map:
\[ \exp: [0,\epsilon)\times S(\nuB ) \ra M
\]
which is a diffeomorphism for $r\neq 0$ in order to endow $[0,\epsilon)\times
S(\nuB )$ with a degenerate metric $g_1$. Clearly there exist a model edge
degenerate metric $h_1$ on $[0,\epsilon)\times S(\nuB )$ of type:
\[dr^2\oplus r^2 g^V\oplus \pi^*g^B
\] 
where $g^V$, the metric on $VS(\nuB )\subset \pi^*\nuB $ is induced by pulling
back the metric $g\bigr|_{\nuB }$. The decomposition is relative to the
Ehresmann connection mentioned earlier.

 We will consider the first-order perturbation
\begin{equation}\label{subvpertu} \hat{h}_1(r)=h_1-2r\mathrm{II}
\end{equation}
where $\mathrm{II}:S(\nu B)\ra \Bil(\pi^*T^*B)$ is the second fundamental form 
\[\mathrm{II}_{W_0}(X_0,Y_0):=g(\nabla_{X_0}Y,W_0), \qquad W_0\in S(\nu_bB),\;\;
X_0,Y_0\in T_bB\]
with $Y$ vector field along $B$ such that $Y(0)=Y_0$. 

\begin{theorem}\label{tvsv}
Let $B\subset M$ be a compact submanifold in a Riemannian manifold $(M,g)$. 
Then the degenerate metric $g_1$ on $[0,\epsilon)\times S(\nuB )$ is a second
order perturbation of the metric $\hat{h}_1$ defined in \eqref{subvpertu}. 
\end{theorem}
\begin{proof} By the Gauss Lemma, $R:=\partial_r$ is a geodesic
field orthogonal to the slices $\{r\}\times S(\nuB )$, and therefore
$g_1=dr^2\oplus g_1(r)$. We need only look at $g_1(r)$ on $T(S\nuB )$. The
metric $g_1(r)$ is obtained  via the map:
\begin{align*} 
\exp^{r}:S \nuB \ra M,&&  (p,v)\ra \exp_p(rv), &&
g_1(r)(\cdot,\cdot):=g(d\exp^r(\cdot),d\exp^r(\cdot)).
\end{align*}
We use  curves  $W:(-\epsilon,\epsilon)\ra S(\nuB )$ with $\gamma(s):=\pi(W(s))$
where $\pi:S(\nuB )\ra B$ is the projection in order to represent tangent
vectors of $S(\nuB )$. Let then
\[ f(r):= g_1(r)(W'_1(0),W_2'(0))=
g\left({\partial_s}\exp^r(W_1(s))_{|s=0},{\partial_s}\exp^r(W_2(s))_{|s=0}
\right).\]
Notice that
\[ J_i(r):={\partial_s}\exp^r(W_i(s))_{|s=0}
\]
are Jacobi vector fields, along the geodesics $r\ra\exp_{\gamma_i(0)}(rW_i(0))$.
We will assume that $W_1(0)=W_2(0)=(b,W_0)\in S(\nu_b B)$. 

It is good to keep in mind that 
\begin{itemize}
\item[(1)] there exists a splitting \[ TS(\nu B)=\pi^* TB\oplus VS(\nu B )  \]
induced by the normal connection $\nabla^{\nu}$ of $\nu B$; consequently, the
derivative $W'(s)$ decomposes as:
\[ W'(s)=(\gamma'(s),(\gamma^*\nabla^{\nu})_{\partial_t}W\bigr|_{t=s})\in
T_{\gamma(s)}B\oplus \nu_{\gamma(s)}B
\]
where $\gamma^*\nabla^{\nu}$ is the pull-back connection on $\gamma^*\nu B\ra
(-\epsilon,\epsilon)$.
\item[(2)] via the same splitting we have an injective morphism of vector
bundles over $S(\nu B)$:
\[ TS(\nu B )\hookrightarrow \pi^* TB\oplus V\nu B=\pi^*TB\oplus \pi^*\nu
B=\pi^*\left(TM\bigr|_{B}\right)
\]
 Hence for $W_0\in S(\nu_b B)$, we have $T_{W_0}S(\nu B)=\{w\in T_bM~|
~g(w,W_0)=0\}$.
\end{itemize}

In order to make the computations more transparent it is useful to  separate
two classes of  vector fields $W$ along $\gamma$.  
\begin{itemize}
\item[(a)] the vertical ones, i.e., those for which $\gamma(s)\equiv b\in B$ is
constant and therefore $J(0)=0$ and $J'(0)=W'(0)\in T_{W_0}S(\nu_{b}B)$ is a
vertical vector in $T_{W_0}S\nuB $.
\item[(b)] the horizontal ones, i.e., those for which $\gamma'(0)\neq 0$ and
$\nabla_{\gamma'}^{\nu}W=(\gamma^*\nabla^{\nu})W\equiv 0$; these are obtained by
parallel transporting the initial vector $W_0$ along $\gamma$ in $\nu B$; notice
that the condition
$\nabla_{\gamma'}^{\nu}W=0$ implies that $W'(0)=(\gamma'(0),0)$ is a horizontal
vector in
$T_{W(0)}S\nuB $ such that $d\pi(W'(0))=\gamma'(0)$; one has $J(0)=\gamma'(0)$ 
and $J'(0):=\nabla_{\partial_r}
J_{|r=0}=\nabla_{\gamma'(0)}W:=(\gamma^*\nabla)_{\partial_t}W\bigr|_{t=0}$ where
$\nabla$ is the Levi-Civita connection on $M$. Since the
$0=\nabla^{\nu}_{\gamma'(0)}W=P^{\nu B}\nabla_{\gamma'(0)}W$ it follows that
$J'(0)$ is a horizontal vector.
\end{itemize} 

By what was just said one has:
\begin{itemize}
\item[(a)]  when $W_1'(0)$, $W_2'(0)$ are both horizontal:
\begin{align} 
f(0)={}&g(J_1(0),J_2(0))=g(\gamma_1'(0),\gamma_2'(0))=g(W_1'(0),W_2'(0)),
\nonumber\\
f'(0)={}&\partial_r g(J_1(r),J_2(r))_{|r=0}=
g(\nabla_{\gamma_1'(0)}W_1,\gamma_2'(0))+g(\gamma_1'(0),\nabla_{\gamma_2'(0)}
W_2)=-2\mathrm{II}_{W_0}(\gamma_1'(0),\gamma_2'(0)),\nonumber\\
f''(0)={}&\partial^2_r
g(J_1(r),J_2(r))_{|r=0}=\left[g(J''_1(r),J_2(r))+2g(J_1'(r),
J_2'(r))+g(J_1(r),J_2''(r))\right]_{|r=0}\nonumber\\
={}&[g\left(R^g(\partial_r,J_1(r))\partial_r,J_2(r)\right)+2g(J_1'(r),
J_2'(r))+g\left(J_1(r),
R^g(\partial_r,J_2(r))\partial_r\right)]_{|r=0} \label{jaceq2}
\end{align}
where we used that $J_1$ and $J_2$ are Jacobi.
\item[(b)]  when $W_1'(0)$ is horizontal and $W_2'(0)$  vertical:
\begin{align*} 
f(0)={}&g(J_1(0),J_2(0))=0,\\
f'(0)={}&g(J'_1(0),J_2(0))+g(J_1(0),J_2'(0))=g(\gamma'_1(0),W_2'(0))=0,\\
f''(0)={}&2g(J_1'(0),
J_2'(0))=0{}&.
\end{align*}
In the last equality we used that in \eqref{jaceq2}, $J_2(0)=0$ and that $J'(0)$
is horizontal while $J_2'(0)$ is vertical.
\item[(c)] when $W_1'(0)$ and $W_2'(0)$ are both vertical:
\begin{align*} 
f(0)={}&0=f'(0),\\
f''(0)={}&2g(J_1'(0),J_2'(0))=2g(W_1'(0),W_2'(0)),\\
f'''(0)={}&\partial_r[g\left(R^g(\partial_r,J_1(r))\partial_r,
J_2(r)\right)+2g(J_1'(r),J_2'(r))+g\left(J_1(r),R^g(\partial_r,
J_2(r))\partial_r\right)]\bigr|_{r=0}=0
\end{align*}
\end{itemize}
 because $J_1(0)=J_2(0)=0$. Summarizing: 
\begin{itemize}
\item for $W_1,W_2\in T_{W_0}S(\nu B)$ both horizontal, 
\begin{align*}g_1(r)(W_1,W_2)={}&g(W_1,W_2)-2r\mathrm{II}_{W_0}(W_1,W_2)+O(r^2);
\end{align*}
\item for $W_1$ horizontal and $W_2$ vertical, $g_1(r)(W_1,W_2)=O(r^3)$;
\item for $W_1$ and $W_2$ both vertical, $g_1(r)(W_1,W_2)=r^2g(W_1,W_2)+O(r^4)$.
\end{itemize}
Recall now that
$g_1^{\varphi}(r)(W_1,W_2)=g_1(r)(P^H(W_1)+r^{-1}P^V(W_1),P^H(W_2)+r^{-1}
P^V(W_2))$. We get that
\[ g_1(r)^{\varphi}_{W_0}(W_1,W_2)=g(W_1,W_2)-2r
\mathrm{II}_{W_0}(P^H(W_1),P^H(W_2))+O(r^2)=\hat{h}_1(r)^{\varphi}(W_1,
W_2)+O(r^2)
\]
and this corresponds to a second-order perturbation of $\hat{h}_1$ according to
Definition \ref{def10}.
\end{proof}

\subsection{Gauss-Bonnet for perturbations of model metrics}\label{GBpmm}
 
 We will look at perturbation of second order (Definition \ref{def10}) of
canonical model edge degenerate metrics. We assume again that $M$ is an edge
manifold.
 
 A canonical model edge degenerate metric $h$ is uniquely determined by the
following data
 \begin{itemize}
 \item[(a)] a collar neighborhood $U\supset \partialM $ with a diffeomorphism
$R:U\ra (-\epsilon,0]\times N$ that makes the obvious diagram commutative;
 \item[(b)] an Ehresmann connection on $\pi:\partialM =N\ra B$;
 \item[(c)] a metric $g^{V}$ on $\Ker d\pi$;
 \item[(d)] a metric $g^B$ on $B$.
 \end{itemize}

Part (a) of the next result justifies (\ref{eqintro3}).
 \begin{theorem} \label{trgl} \begin{itemize}
\item[(a)] Let $g$ be a first-order perturbation of a canonical model edge
metric $h$. Then $\partial_r\in \Gamma(TM\bigr|_{\partial M})$ is the exterior
normal unit of $\partial M$ with respect to $g^{\varphi}$, the Pfaffian $\Pf^g$
is a smooth form on $M$ and 
\begin{equation}\label{eq88} \lim_{r \ra 0}\int_{\{r\}\times
N}\TPf^g=\TPf(\varphi\nabla^g\varphi^{-1}\bigr|_{r=0},\partial_r).
\end{equation}
 Hence the following holds:
\[ (2\pi)^k\chi(M)=\int_{M}\Pf^{g}-\int_{B}\left(\int_{\partial
M/B}\TPf(\varphi\nabla^g\varphi^{-1}\bigr|_{r=0},\partial_r)\right).
\]

\item[(b)] Suppose that $g$ is a second-order perturbation of a canonical model
edge degenerate metric $h$. Then 
 \begin{equation} \label{eq89}\lim_{r \ra 0}\int_{\{r\}\times
N}\TPf^g=\lim_{r\ra 0}\int_{\{r\}\times N}\TPf^h.
 \end{equation}
 Consequently, the Gauss-Bonnet formula of Theorem \ref{Theorem1} holds verbatim
where the
 odd Pfaffian form is associated to the degenerate metric $h$. 
 \end{itemize}
 \end{theorem}
 \begin{proof} We use the notations of Section \ref{Sec8}.  One consequence of
the definition of perturbation is that the bilinear form
 \[ g^{\varphi}(\cdot,\cdot)=g(\varphi^{-1}(\cdot),\varphi^{-1}(\cdot))
 \]
 is a well-defined smooth metric on $TM$. Due to the fact that
$g^{\varphi}\bigr|_{TM\bigr|_{\partial
M}}=h^{\varphi}\bigr|_{TM\bigr|_{\partialM }}$, the vector $\partial_r$ has norm
$1$ also for $g^{\varphi}$ at $r=0$.
 
  Moreover if $\nabla^g$ is the
Levi-Civita connection of $g$ away from $r=0$, then
$\varphi\nabla^g\varphi^{-1}$ is a $g^{\varphi}$-metric connection. As proved in
Theorem \ref{Te1} this connection is defined everywhere and therefore $\Pf^g$ is
smooth on $M$.
 
It is easy to check that if $\nabla^1$ and $\nabla^2$ are two $g$-metric
compatible connections and  $\varphi:E\ra E$ is a bundle isometry where on the
right one uses $g^{\varphi}$ then
\[ \TPf(\nabla^1,\nabla^2)=
\TPf(\varphi\nabla^1\varphi^{-1},\varphi\nabla^2\varphi^{-1}).\]
This is the case for $E=TM\bigr|_{\{r\}\times N}$ with $r\neq 0$ and
$\nabla^1=d\oplus P\nabla^g P$ and $\nabla^2=\nabla^g$ constructed as in Example
\ref{ex1}, where $P$ is the $g$-orthogonal projection onto $T(\{r\}\times
N)\subset E$.  The fact that $\varphi\nabla^g\varphi^{-1}$ exists for all values
of $r$ implies immediately that the limit in \eqref{eq88}
exists. 

Moreover the limit is entirely determined by
$\varphi\nabla^g\varphi^{-1}\bigr|_{TM\bigr|_{\partialM }}$ and the orthogonal
decomposition
\[ TM\bigr|_{\partialM }=\bR\partial_r\oplus TN.
\] 
In fact, if $s_g(r)$ is the $g$-exterior unit normal to $\{r\}\times N$ then
$\varphi(s_g(r))$ is the $g^{\varphi}$ unit normal and it is easy to check that
$\varphi(s_g(r))$ is a parallel section with respect to
$\varphi\nabla^1\varphi^{-1}$ and therefore $\varphi\nabla^1\varphi^{-1}$, being
$g^{\varphi}$-compatible has a diagonal block-decomposition with respect to
$TM\bigr|_{\{r\}\times N}\simeq \bR\varphi(s_g(r))\oplus T(\{r\}\times N)$ of
type:
\[\varphi\nabla^1\varphi^{-1}= d\oplus
P^{\varphi}\varphi\nabla^g\varphi^{-1}P^{\varphi}
\]
 where $P^{\varphi}$ is the $g^{\varphi}$-orthogonal projection onto
$T(\{r\}\times N)$ and $\varphi$ takes the $g$ unit normal to $\{r\}\times N$ to
the $g^{\varphi}$ unit normal to $\{r\}\times N$.
 
Therefore
 \[ \TPf^g\bigr|_{\{r\}\times
N}=\TPf\left(\varphi\nabla^g\varphi^{-1}\bigr|_{\{r\}\times N},\varphi
(s_g(r))\right)
 \]
and one gets (\ref{eq88}) when $r\ra 0$.
 
For $(b)$ recall that for second-order perturbations
$\varphi\nabla^g\varphi^{-1}\bigr|_{r=0}=\varphi\nabla^{h}\varphi^{-1}\bigr|_{
r=0}$.
\end{proof}

\subsection{First order perturbations}

One can obtain the following alternative form of Theorem \ref{trgl} part (a),
stated as Theorem \ref{Theorem5} in the Introduction:
 \begin{theorem}\label{Tcorfac} Let $g$ be a first-order perturbation of a model
edge metric $h=dr^2\oplus r^2 h^{V}\oplus \pi^*h^B$. Then 
 \[
(2\pi)^k\chi(M)=\int_{M}\Pf^g-\sum_{j=0}^{(f-1)/2}{\tilde{c}\left(\tfrac{f-1}{2}
-j\right)}\int_B\left(
{\Pf(h^B)}\int_{N/B}P_{j,f}\left(h^{V}\right)\right)-\int_{\partialM }
\TPf(\nabla^h_1,\nabla^g_1)
\]
where $\nabla^h_1=\varphi\nabla^h\varphi^{-1}\bigr|_{r=0}$ is described in
\eqref{extat0} and $\nabla^g_1=\varphi\nabla^g\varphi^{-1}\bigr|_{r=0}$. The 
form $\Pf(h^B)$ is zero, by definition, when $\dim{B}$ is odd.
 \end{theorem}
 \begin{proof} We use the notation of the proof of Theorem \ref{trgl}. As in
Subsection \ref{sbs2}, we consider the metric
$ds^2+(1-s)h^{\varphi}+sg^{\varphi}$ on
$[0,1]\times U$, where $U$ is the collar.
Parallel transport induces a bundle isometry  
\[ \tau_1^{-1}:(TM\bigr|_{U}, g^{\varphi})\ra (TM\bigr|_{U}, h^{\varphi}).\]
While parallel transport $\tau_1^{-1}$ need not take $\varphi(s_g(r))$ to
$\varphi(s_h(r))$  since
at $r=0$ $(\varphi s_g(r))\bigr|_{r=0}=(\varphi s_h(r))\bigr|_{r=0}=\partial_r$
and $\tau_1\bigr|_{r=0}=\id$ it is
clear that for $r$ small one can find a smooth homotopy between
$\tau_1^{-1}\circ \varphi(s_g(r))$  and $\varphi(s_h(r))$ within
$(S(TM),h^{\varphi})$. Then we
can apply Proposition \ref{p8} to conclude that
\begin{align*} 
\lefteqn{\TPf(\varphi\nabla^g\varphi^{-1},
\varphi(s_g(r)))-\TPf(\varphi\nabla^h\varphi^{-1},\varphi(s_h(r)))}\\
{}&=\TPf(\tau^{-1}_1\varphi\nabla^g\varphi^{-1}\tau_1,\tau_1^{-1}
(\varphi(s_g(r))))-\TPf(\varphi\nabla^h\varphi^{-1},\varphi(s_h(r)))\\ 
{}&=-\TPf(\tau^{-1}_1\varphi\nabla^g\varphi^{-1}\tau_1,\varphi\nabla^h\varphi^{
-1})\bigr|_{\{r\}\times N}+d\gamma.
\end{align*}
From here we deduce immediately that at $r=0$
\[ \int_{\partial
M}\TPf(\varphi\nabla^g\varphi^{-1}\bigr|_{r=0},\partial_r)=\int_{\partial
M}\TPf(\varphi\nabla^h\varphi^{-1}\bigr|_{r=0},\partial_r)+\int_{\partial
M}\TPf(\nabla_1^h,\nabla_1^g).
\]
 \end{proof}

 \begin{rem} For horizontal variations of the metric that are constant along the
fiber, more precise computations are given in Theorem \ref{olst}.
 \end{rem}
The sum in the previous Theorem also has an alternative characterization.
 
\begin{prop} \label{Tcorfac1} 
If $\nabla'$ is the connection from \eqref{nabla'}, then $\varphi
\nabla'\varphi^{-1}=d\oplus \pi_2^*\nabla^{\VN }\oplus \pi_2^*\pi^*\nabla^B$ and
\begin{equation}\label{trgl4}\sum_{j=0}^{(f-1)/2}{\tilde{c}\left(\tfrac{f-1}{2}
-j\right)}
\int_B\left( {\Pf(h^B)}\int_{N/B}P_{j,f}\left(h^{V}\right)\right)
=\int_{\{0\}\times N}\TPf(\varphi
\nabla'\varphi^{-1},\varphi\nabla^h\varphi^{-1}).
\end{equation}
\end{prop}
\begin{proof} The first statement is a simple computation. The left hand side of
(\ref{trgl4}) is equal, by the proof of Theorem \ref{Theorem1} to
 \[\lim_{r\ra 0}\int_{\{r\}\times N}
\TPf\left(\nabla^h,s_h(r)\right)=\int_{\partialM
}\TPf(\varphi\nabla^h\varphi^{-1}
 \bigr|_{r=0},\partial_r).\]
 From \eqref{extat0}  we see that at $r=0$ the projection of 
$\varphi\nabla^h\varphi^{-1}$ on $TN$ gives the connection $ \pi_2^*\nabla^{\VN
}\oplus \pi_2^*\pi^*\nabla^B$, hence the induced block-diagonal connection from
$\varphi\nabla^h\varphi^{-1}$ with respect to the decomposition
$TM=\bR\partial_r\oplus TN$ (as in Example \ref{ex1}) is exactly
$\varphi\nabla'\varphi^{-1}$. Hence, by definition:
 \[\TPf(\varphi
\nabla'\varphi^{-1}\bigr|_{r=0},\varphi\nabla^h\varphi^{-1}\bigr|_{r=0}
)=\TPf(\varphi\nabla^h\varphi^{-1}
 \bigr|_{r=0},\partial_r)
 \]
\end{proof}

We use this result in order to give a more geometric expression to the boundary
contribution for first-order perturbations of conical model metrics.

\begin{definition} Let $g$ be a first-order perturbation of the \emph{conical}
metric 
\[ h=dr^2\oplus r^2g^N
\]
on $(-\epsilon,0]\times N$. Define the (asymptotic) second fundamental form
$\II^g$ of $\partialM :=\{0\}\times N$ as follows:
\[\II^g(X,Y):=
h^{\varphi}\left((\nabla_1^g)_XY,\partial_r\right)=g^{\varphi}
\left((\nabla_1^g)_XY,\partial_r\right)
\]
where $\nabla_1^g=\varphi\nabla^g\varphi^{-1}\bigr|_{TM\bigr|_{\partialM }}$ is
the connection resulting from Theorem \ref{Te1}.
\end{definition}
Denote by $R^N$ the curvature form of the metric $g^N$ and set
\[\mathcal{G}_{j,2k-1}^{\partialM
}:=\frac{1}{j!(2k-1-2j)!}\cB_{h^{\varphi}}\left((R^N)^j\wedge
(\II^g)^{2k-1-2j}\right).
\]
\begin{theorem} For first-order perturbations $g$ of conical metrics $dr^2\oplus
r^2g^N$ the following holds
\[ (2\pi)^k\chi(M)=\int_{M}\Pf^g-\sum_{j=0}^{k-1}(-1)^j(2j-1)!!\int_{\partialM
}\mathcal{G}_{k-1-j,2k-1}^{\partialM }.
\]
\end{theorem}
\begin{proof} Proposition \ref{Tcorfac1} and Theorem \ref{Tcorfac} together say
that the contribution of the boundary is
\[\int_{\partialM } \TPf(\varphi
\nabla'\varphi^{-1},\varphi\nabla^h\varphi^{-1})+\int_{\partialM
}\TPf(\varphi\nabla^h\varphi^{-1},\varphi\nabla^g\varphi^{-1})=\int_{\partialM }
\TPf(\varphi\nabla'\varphi^{-1},\varphi\nabla^g\varphi^{-1}).
\]
In the conical case
\[\varphi\nabla'\varphi^{-1}=d\oplus \pi_2^*\nabla^N.
\]
and these are also the block-diagonal components of
$\varphi\nabla^g\varphi^{-1}$ at $r=0$. In order to justify this let us take
another look at \eqref{eq100}. When $X$ is tangent to $\partialM $ then since
$C^{\varphi}$ is the identity on $\partialM $ we get that at $r=0$ one has
\[h^{\varphi}\left((\nabla^{\varphi}_X C^{\varphi})(Y'),Z'\right)=0.
\]
On the other hand, $\varphi^{-1}(Y')=r^{-1}Y'$ and $\varphi(X)=rX$ for $X,Y'\in
\Gamma(TN)$. Then the factors $r^{-1}$ and $r$ cancel each other out and one
has:
\[h^{\varphi}\left(
(\nabla_{\varphi^{-1}(Y')}^{\varphi}C^{\varphi})(\varphi(X))),Z'\right)=0.
\]
The last term of \eqref{eq100} is similar and therefore also vanishes. We
conclude that for $X\in T\partialM $, $Y',Z'\in T\partialM $ 
\[ h^{\varphi}(\omega^{\varphi}(X)(Y'),Z')=0.
\]
One sees easily that the same holds for $Y',Z'=\partial_r$. This justifies the
claim that the block diagonal components of $\varphi\nabla^g\varphi^{-1}$ and
$\varphi \nabla^h\varphi^{-1}$ when restricted to $r=0$ are the same. But the
block diagonal components of $\varphi \nabla^h\varphi^{-1}$ are the same as
those of $\varphi\nabla'\varphi^{-1}$.
 Finally, the off-diagonal components of $\varphi\nabla^g\varphi^{-1}$ are
precisely the components of $\II^g$.
 
 The situation is similar now to the proof of the Gauss-Bonnet formula
\ref{eq-4}, 
and the transgression
$\TPf(\varphi\nabla'\varphi^{-1},\varphi\nabla^g\varphi^{-1})$ can be computed
accordingly.
\end{proof}

\section{Perturbations of manifolds with fibered boundary} \label{Secaden}

Recall that an end of a manifold with fibered boundary is modeled on
$(1,\infty)\times N$ with the metric
\[ dr^2\oplus g^V \oplus r^2\pi^*g^B
\]
It is convenient to set $u=r^{-1}$. In the new coordinate, the metric
on $U^c=(0,1)\times N$ is of type:
\[ h= (d(u^{-1}))^2\oplus g^V\oplus u^{-2}\pi^*g^B.
\]
This leads us to consider, in the spirit of the previous section, the
following endomorphism $\varphi:\bR\oplus TN\bigr|_{U}$:
\begin{align*} 
\varphi(s,v,w)=(u^{-2}s,v,u^{-1}w),&& s\in \bR,\; v\in \pi_2^*\VN ,\; w\in
\pi_2^*\pi^*TB
\end{align*}
where we use $\partial_u$ as the coordinate vector on $\bR$. Then clearly
\[ h^{\varphi}(X',Y'):=h(\varphi^{-1}(X'),\varphi^{-1}(Y'))
\]
extends to a smooth metric on $(-1,0]\times N=:U$. 
\begin{theorem} \label{Te2} 
Let $\nabla^h$ be the Levi-Civita connection of $h$
on $U^c$. Then $\varphi\nabla^h\varphi^{-1}$ extends to a smooth connection on
$U$ compatible with $h^{\varphi}$.
\end{theorem}
\begin{proof} 
The proof follows closely that of Theorem \ref{fundprop}. 
The auxiliary connection $\nabla'$ is
\[ \nabla'=\left[d-\frac{2}{u}du\right]\oplus \pi_2^*\nabla^{\VN }\oplus
\left[\left(\frac{\partial}{\partial
u}-\frac{1}{u}\right)du+\pi_2^*\pi^*\nabla^B\right]
\]
where $d$ is the trivial connection on $\bR$ and $\nabla^{\VN }$ is the
projection
of the Levi-Civita connection of a slice $u=\text{constant}$ to $\VN $. 
We notice that
\begin{itemize}
\item[(a)] $\varphi\nabla'\varphi^{-1}$ extends smoothly;
\item[(b)] $d-\frac{2}{u}du$ and $\pi_2^*\nabla^{\VN }$ are the projections of
the
Levi-Civita connection $\nabla^h$ to $\bR$ and to $\pi_2^*\VN $ respectively.
\item[(c)] $\partial_u$ is  orthogonal to the slices and the unit normal vector
is $u^2\partial_u$; the vector field $X=u^2\partial_u$ satisfies the conditions
of Lemma \ref{L1} and this allows the computation of the second fundamental form
of the slices in the same vein we did before.
\end{itemize}
One  then carefully analyzes the blocks of the $1$-form $\varphi
(\nabla^h-\nabla')\varphi^{-1}$ and sees that they extend as well. 
\end{proof}
\begin{rem} One might prefer to work directly with the $r$ coordinate. In that
case one first needs to turn $(1,\infty]$ into a manifold and this can be done
via the unique chart $(-1,0]\ra (1,\infty]$ where $u\mapsto -1/u$ for $u\neq 0$ and
$0\ra \infty$. Then the vector field that trivializes the tangent bundle of
$(1,\infty]$ (using the standard coordinate of $(1,\infty$) is
$\widetilde{\partial_r}:=r^2\partial_r$ which makes sense also at $\infty$ and
corresponds to $\partial_u$. Consequently the metric on $(1,\infty)\times N$ in
these coordinates can be written as $r^4\widetilde{dr}^2\oplus g^V\oplus
r^2g^B$ and $\varphi(s,v,w)=(r^2s,v,rw)$, etc.
\end{rem}

\begin{definition}\label{pertdef2}
A  perturbation of the model fibered boundary metric $h$ is a metric $g$ such
that $g^{\varphi}$ extends smoothly
to a metric on $TM\bigr|_{U}$ and 
\[ g^{\varphi}=h^{\varphi}+f(u)\alpha
\]
for some smooth function $f$ on $(-1,0]$ that vanishes at $0$, and some smooth
bilinear symmetric form $\alpha$ on $TM\bigr|_{U}$. It is called of first,
respectively second order if $f(u)=O(u)$, respectively $f(u)=O(u^2)$.
\end{definition}

\begin{lem} A perturbation of first, resp.\ second order for the metric
$h=dr^2\oplus r^2g^N$ on $(1,\infty)\times N$ is a metric $g$ such that
\[ g=h+O(r^{-1})\cdot \gamma_N(r) dr^2+O(1)\cdot (dr\otimes
\beta_N(r)+\beta_N(r)\otimes dr)+O(r)\cdot \alpha_N(r)
\]
respectively
\[g=h+O(r^{-2})\cdot  \gamma_N(r) dr^2+O(r^{-1}) \cdot (dr\otimes \beta_N(r)+
\beta_N(r)\otimes dr)+O(1)\cdot\alpha_N(r)
\]
where $\gamma_N(r)$, $\beta_N(r)\in \Omega^1(N)$ and $\alpha_N(r)\in
\Gamma^+(T^*N\otimes T^*N)$ are smooth families of $0$ and $1$-forms, resp.\
symmetric $(1,1)$ double forms on $N$ which extend smoothly at $\infty$, i.e.,
when composed with $-1/u$ they extend smoothly to $u=0$.
\end{lem}
\begin{proof} Straightforward.
\end{proof}

\begin{theorem} For a first-order perturbation $g$ of $h$, the connection
$\varphi\nabla^g\varphi^{-1}$ extends to a smooth connection, while for a second
order perturbation the restriction of $\varphi\nabla^g\varphi^{-1}$ to $u=0$ (or
$r=\infty$) coincides with the restriction of $\varphi\nabla^h\varphi^{-1}$ to $u=0$.
\end{theorem}
\begin{proof} Almost  identical to Theorem \ref{Te1}. Notice that in formula
\eqref{eq100}, $\nabla^{\varphi}_{\varphi^{-1}(Y')}C^{\varphi}$ makes sense at
$u=0$ as $\varphi^{-1}(s,v,w)=(u^2s,v,uw)$ while $\nabla^{\varphi}$ and
$C^{\varphi}$ extend by Theorem \ref{Te2} and Definition \ref{pertdef2} respectively.
\end{proof}
\begin{cor} The Gauss-Bonnet formula of Theorem \ref{th1.3} holds 
for second-order perturbations  
of a metric with fibered boundary. 
\end{cor}
\begin{example} A catenoid in $\bR^3$ has the following
parametrization 
\[ \mathcal{C}=\{ (\cosh{(v)}\theta,v)\in \bR^3~|~\theta \in S^1,\; v\in \bR\}.
\]
Use the change of coordinates $v=\arcsinh(r)$ in order to write the metric as 
\[ dr^2+(1+r^2)d\theta^2\]
where $\partial_{\theta}$ is the unit tangent vector on $S^1$ with the round
metric. This is a second-order perturbation of the flat metric
$dr^2+r^2d\theta^2$. The catenoid is a minimal surface with two ends, 
its total Gaussian
curvature is $-4\pi$, Euler characteristic $0$, and each end
contributes to the Gauss-Bonnet formula by $1$, which is the integral of the odd 
Pffafian $(2\pi)^{-1}\TPf(S^1,g_{\round},1)$.
\end{example}

\section{Riemannian orbifolds with simple singularities}

Let $M$ be a Riemannian manifold and suppose $G$ is a finite group that acts  by
isometries on $M$ such that the following properties are satisfied:
\begin{itemize}
\item[(i)] $\Fix_G(M)$ is a (necessarily closed) submanifold of $M$;
\item[(ii)] $G$ acts freely on $M\setminus \Fix_G(M)$.
\end{itemize}
The quotient $\hat{M}:=M/G$ is an example of a Riemannian orbifold. We use the
following definition (see \cite{Bor}):
\begin{definition} A Riemannian orbifold $\hat{M}$ is a Hausdorff topological
space endowed with a countable basis of open charts $U_i$, closed under finite
intersection such that each chart $U_i$ is  homeomorphic with  the quotient of
an open set $\tilde{U}_i\subset \bR^n$ endowed with a Riemannian metric $g_i$
(that turns $\tilde{U}_i$ into a geodesically  convex set) modulo the action of
a finite group $G_i$ that acts effectively by isometries on $\tilde{U}_i$.
Moreover, the following data  is part of the structure:

For  each inclusion $\iota:U_i\subset U_j$ there exist
\begin{itemize}
\item[(i)] an injective group morphism $\phi_{ij}:G_i\ra G_j$;
\item[(ii)] an isometric embedding $\tilde{f}_{ij}:\tilde{U}_i\ra \tilde{U}_j$,
equivariant with respect to $\phi_{ij}$
\end{itemize}
fitting a commutative diagram
\[ \xymatrix{   
\tilde{U}_i \ar[d]^{\tilde{f}_{ij}} \ar[r] & 
\tilde{U}_i/G_i \ar[d]^{\tilde{f}_{ij}/G_i}&&  
U_i \ar[ll]_{\sim}\ar@{_{(}->}[d]^{\iota} \\
\tilde{U}_j \ar[r] & 
\tilde{U}_j/\phi_{ij}(G_i)\ar[r]& 
\tilde{U}_j/G_j & 
U_j \ar[l]_{\sim}}
\]
\end{definition}

Clearly, every open subset of an orbifold is an orbifold. 
\begin{definition} Let $M$ and $N$ be two Riemannian orbifolds. Then a
homeomorphism $f:M\ra N$ is an isometry if it is a local isometry, i.e., if for
every pair $m\in M$, $n=f(m)\in N$ there exist
\begin{itemize}
\item[(a)] charts $m\in U\subset M$, $n\in D\subset N$ with corresponding open
sets $\tilde{U}\subset \bR^n$ and $\tilde{D}\subset \bR^n$ and groups $G_U$ and
$G_D$
\item[(b)] a group isomorphism $\phi: G_U\ra G_D$, and
\item[(c)] an isometry $\tilde{f}:\tilde{U}\ra \tilde{D}$ which is equivariant
with respect to $\phi$
\end{itemize}
such that the next diagram commutes
\[ \xymatrix{\tilde{U}\ar[d]^{\tilde{f}}\ar[r] & 
\tilde{U}/G_{U}\ar[d]^{\tilde{f}/G_U} \ar[r]^{\sim} &   
U\ar[d]^f
\\
\tilde{D}\ar[r] &
\tilde{D}/G_{D} \ar[r]^{\sim}& 
D}
\]
\end{definition}

For every point $p\in M$, the isomorphism class of the isotropy group $G_p$ is
unambiguously defined. In a chart $U_i\ni p$ the  group $G_p$ is represented by
the conjugacy class of the isotropy group  of a lift $\tilde{p}\in \tilde{U}_i$.
\begin{definition}
The singular locus $Z$ of an orbifold is:
\[ Z:=\{p\in M~|~G_p\neq \{e\}\}.
\]
\end{definition}

From the above definitions it is clear that $\hat{M}\setminus Z$ inherits a
Riemannian manifold structure and we denote the metric by $g$.
We will consider Riemannian orbifolds $\hat{M}$ for which the singular strata
have a ``nice'' structure.

\begin{definition} \label{defRorb}
A Riemannian orbifold $\hat{M}$ is called with \emph{simple singularities} if
each connected component $Z_i$ of $Z$ has the property that there exists 

\begin{itemize} 
\item an open neighborhood $D_i$ of $Z_i$,
\item a finite group $\Gamma_i$, and
\item a Riemannian manifold $M_i$
\end{itemize} such that
\begin{itemize}
\item[(i)] $\Gamma_i$ acts by isometries on $M_i$, $\Fix_{\Gamma_i}(M_i)$ is a
compact submanifold in $M_i$ and $\Gamma_i$ acts freely on $M_i\setminus
\Fix_{\Gamma_i}(M_i)$; 
\item[(ii)] There exists an isometry of Riemannian orbifolds $h_i:D_i\ra
M_i/\Gamma_i$ such that 
\[h_i(Z_i)=\Fix_{G_i}(M_i).\] 
\end{itemize}
\end{definition}

Any Riemannian orbifold with isolated singularities satisfies
the previous definition.
Denote by $\Fix(\hat{M})$ the set of connected components of the singular locus
$Z$.

\begin{theorem} \label{thro}
Let $\hat{M}$ be a compact Riemannian orbifold with simple
singularities of dimension $2k$ and let $g$ be the Riemannian metric on
$\hat{M}\setminus Z$. Then 
\begin{equation}\label{fineq}
\chi(\hat{M})=\frac{1}{(2\pi)^k}\int_{\Int{\hat{M}}}   \Pf^g + \sum_{Z_i\in
\Fix(\hat{M})} \chi(Z_i)\frac{|G_i|-1}{|G_i|}.
\end{equation}
\end{theorem}
\begin{proof}  Fix a connected component $Z\in \Fix(\hat{M})$ and let $D$ be the
neighborhood of $Z$ from Definition \ref{defRorb} such that $D\simeq M/\Gamma$.
Let $B:=\Fix_{\Gamma}(M)$. Since the action of $\Gamma$ on  $M$ is via
isometries in the induced action $\Gamma\times TM\ra TM$ via differentials, the
subset $S(\nuB )$ is invariant. 
Moreover, the action is free and linear in every fiber $S(\nu_b B)$.

Now let $\Gamma$ act trivially on $(-\epsilon,0]$. Then it is  straightforward
to see that
\begin{align*} 
\exp:(-\epsilon,0]\times S(\nuB )\ra M,&& (r,p,v)\ra \exp_p(rv)
\end{align*}
is a $\Gamma$-equivariant map since every isometry $g\in \Gamma$ will take a
geodesic with initial conditions $(p,v)$ to a geodesic with initial conditions
$(gp, d_pg(v))$.

It follows that we can find an (equivariant) tubular neighborhood for every
$Z_i\in \Fix(\hat{M})$ whose boundary is diffeomorphic to a quotient 
$N=S(\nuB_i)/\Gamma_i$. By Mayer -Vietoris, 
\begin{equation}\label{trueform} \chi(\hat{M})=\chi\left(\hat{M}\setminus
Z\right) +\sum_{Z_i\in \Fix(\hat{M})} \chi(Z_i)
\end{equation}

One
applies Gauss-Bonnet for manifolds with boundary in the complement of these
tubular neighborhoods in $\hat{M}$ and then passes to limit $r\ra 0$ in order to
obtain a formula for $\chi\left(\hat{M}\setminus Z\right)$. We
can therefore restrict our attention to what happens in the neighborhood $D$ with
the limits of the integrals of the transgression forms.

Recall now that the manifold $\tilde{M}:=(-\epsilon,0]\times S(\nuB )$ has a
model degenerate metric, left invariant by $\Gamma$ (it is obvious that $\Gamma$ leaves invariant the splitting
$TS(\nuB )=VS(\nuB )\oplus HS(\nuB )$).

Assume first that the exponential map
$\exp:D_{\epsilon}(\nuB )\ra M$ is an isometry onto its image. Then the induced
map:
\[ \exp/\Gamma:(D_{\epsilon}(\nuB )/\Gamma)\setminus{\{0\}}\ra
(M/\Gamma)\setminus B
\]
is an isometry onto its image where $\{0\}$ is the zero section of the disk
bundle $D(\nuB )$. 

Use Theorem \ref{Theorem1}, Examples \ref{examp1}  together with
(\ref{trueform}) in order to
conclude that formula \eqref{fineq} holds in this case since the fiber-integral 
equals the integral over the Riemannian manifold $S(\nu_b B)/\Gamma$ of the 
integrand that appears in \eqref{trsp}. That integrand is invariant under the
action of rotations and therefore descends to $S(\nu_b B)/\Gamma$. The result of
fiber integration is therefore $\frac{1}{|\Gamma|}$.

In the general case (without any assumption about the exponential map), by Theorem \ref{tvsv}, 
the degenerate metric $g$ on $(-\epsilon,0]\times S(\nuB )$ 
is a first-order perturbation of the degenerate model metric. It is easy to see
that the transgression form on the slice $\{r\}\times S(\nu B)/\Gamma$,
$r\neq0$ pulls-back to the transgression form induced by $g$ on $\{r\}\times
S(\nu B)$. Since $\Gamma$ acts freely, the map $\{r\}\times S(\nu
B)\ra\{r\}\times S(\nu B)/\Gamma$ is a covering with $|\Gamma|$ sheets.
Therefore in the limit $r\ra 0$, the integral that interests us amounts to
$\frac{1}{|\Gamma|}$ of the corresponding integral over $S(\nu B)$. But the latter
equals $\chi(B)$, by the concluding remarks of Example \ref{examp1}.
\end{proof}

\section{Applications}

\begin{cor}
Let $\hat{M}$ be a compact Riemannian orbifold with simple
singularities of dimension $2k$ and let $g$ be the Riemannian metric on
$\hat{M}\setminus Z$. Then $(2\pi)^{-k}\int_{\Int{\hat{M}}}\Pf^g$ is
rational.
\end{cor}
This follows immediately from theorem \ref{thro}. If the orbifold $\hat{M}$ is
the finite quotient of a closed smooth manifold $X$, one can obtain this result
from the Gauss-Bonnet formula on $X$, however such a $X$ does not exist in
general.

The Gauss-Bonnet formul{\ae} proved here imply some global obstructions for the 
existence of flat cobordisms with prescribed ends of fibered boundary- or
incomplete edge type. 
The simplest instance of such an obstruction arises for even-dimensional cones
modeled by quotients of the round sphere, for instance lens spaces. 
\begin{cor}
There do not exist flat metrics on a compact manifold with one cone singularity
modeled on 
$\Gamma\backslash S^{2k-1}$ for a nontrivial group of isometries $\Gamma$ acting
freely on the round sphere.
\end{cor}
\begin{proof}
When we remove a point from a smooth 
manifold $M$, the Euler characteristic decreases by $1$, and this is reflected
in the transgression form of Theorem \ref{Tconca} on the odd round sphere: the
integral of this local transgression form must equal $(2\pi)^k$ (Remark \ref{nrem}).
We deduce that on the quotient of $S^{2k-1}$ by a finite group of isometries
$\Gamma$ acting freely, this transgression form integrates to $(2\pi)^k/|\Gamma|$.
The Pfaffian form of a flat metric vanishes, hence $1/|\Gamma|=\chi(M)\in{\mathbb Z}$,
thus $\Gamma$ must be trivial.
\end{proof}

More generally, for edge metrics Theorems \ref{Theorem1} and \ref{trgl} imply
some restrictions for the existence of a flat manifold bounding an edge
singularity modeled on spherical fibrations:
\begin{cor}\label{eo}
Let $N\to B$ be a locally trivial fibration of closed manifolds with fiber type
$F$. 
If there exists a compact flat Riemannian 
manifold $(M,g)$ bounding $N$ endowed with a second-order perturbation
of a model edge metric \eqref{eq-3} where all the fibers have constant sectional curvature $1$, 
then the order of $\pi_1(F)$ must divide $\chi(B)$.
\end{cor}
\begin{proof}
Each fiber is isometric to the quotient of the round sphere by the free
action of a finite group $\Gamma=\pi_1(F)$ of isometries of $S^{2f-1}$, hence
by Remark \ref{nrem} the
integral of the transgression form along each fiber is constant equal 
to $(2\pi)^f/|\Gamma|$. 

The conclusion follows from this remark and from the Gauss-Bonnet formula of 
Theorem \ref{Theorem1}, which by Theorem \ref{trgl} remains 
valid also for second-order perturbations of model edge singularities.
Of course, the Pfaffian term vanishes by the flatness assumption on $g$.
\end{proof}

Finally, exactly the same argument using Theorem \ref{th1.3} instead of \ref{Theorem1}
implies an obstruction for the existence of flat
manifolds with fibered boundary ends:
\begin{cor}
Assume that $(M,g)$ is a flat manifold which near the boundary $N$
is a second-order perturbation of a fibered
boundary metric modeled by a fibration $N\to B$, where $B$ is 
the quotient of the round sphere $S^{2b-1}$ by the free action of a finite group $\Gamma$
of isometries. Then the order of $\Gamma$ must divide $\chi(F)$.
\end{cor}
The proof is identical to that of Corollary \ref{eo}, applying Theorem \ref{th1.3} 
instead of \ref{Theorem1}.

\end{document}